\documentclass{article}%
\usepackage{eurosym}
\usepackage{amsmath}
\usepackage{amsfonts}
\usepackage{graphicx}
\usepackage{amsfonts}
\usepackage{mathrsfs}
\usepackage{amssymb}%
\usepackage{xcolor}
\usepackage{exscale}
\usepackage{latexsym}
\usepackage{mathtools}
\newenvironment{proof}[1][Proof]{\textbf{#1.} }{\ \rule{0.5em}{0.5em}}
\usepackage{srcltx}
\usepackage{esint}
 \numberwithin{equation}{section}
\usepackage{amsfonts,amssymb,amscd,amsmath}
\everymath{\displaystyle}
\newtheorem{theorem}{Theorem}[section]
\newtheorem{corollary}[theorem]{Corollary}
\newtheorem{lemma}[theorem]{Lemma}

\newtheorem{remark}{Remark}

\newtheorem{thm}{Theorem}[section]

\newtheorem{defn}[thm]{Definition}
\mathchardef\mhyphen="2D

\newcommand{\RR}{\mathbb{R}}
\newcommand{\mm}{\mathfrak{M}}
\usepackage[margin=1.27in]{geometry}
\newcommand{\avg}[2]{\langle #1 \rangle_{#2}}

\begin{document}

\title{Gradient weighted norm inequalities 
 for linear elliptic equations with discontinuous coefficients}

\author{ {Karthik Adimurthi}\\
Department of Mathematics, Louisiana State University\\
\and
{Tadele Mengesha}\\
Department of Mathematics,
University of Tennessee
\and
{Nguyen Cong Phuc }\\
Department of Mathematics,
Louisiana State University
}\maketitle

%

\begin{abstract}
Local and global weighted norm estimates involving Muckenhoupt weights   are obtained for gradient of solutions to linear elliptic Dirichlet boundary value problems in divergence form over a Lipschitz domain $\Omega$.  The gradient estimates are obtained in weighted Lebesgue and Lorentz  spaces, which also yield estimates in Lorentz-Morrey spaces as well as H\"older continuity of solutions.
The significance of the work lies on its applicability  to very weak solutions (that belong to $W^{1,p}_{0}(\Omega)$ for some $p>1$ but not necessarily in  $W^{1,2}_{0}(\Omega)$) to inhomogeneous equations with coefficients that may have discontinuities but 
have a small mean oscillation.
The domain  is assumed to have a Lipschitz boundary with small Lipschitz constant and as such allows corners. The approach implemented makes use of  localized sharp maximal function estimates as well as known regularity estimates for very weak solutions to the associated  homogeneous equations. 
The estimates are optimal in the sense that they coincide  with classical weighted gradient estimates in the event the coefficients are continuous and the domain has smooth boundary.  

\end{abstract}

\section{Introduction}
 \subsection*{Perspective and description of main results}
In this paper, we obtain weighted gradient estimates for  weak solutions $u$ of 
\begin{equation}
\label{basic-pde}
\begin{aligned}
\text{div}\, \mathbb{A}(x) \nabla u &=  \text{div}\, {\bf f} \quad \text{in}\, \Omega, \quad 
u = 0  \quad \text{on} \,\partial \Omega
\end{aligned}
\end{equation}
for a class of possibly discontinuous coefficients $\mathbb{A}$  and bounded domains $\Omega\subset\RR^n$ with corners. 
The matrix of coefficients $\mathbb{A}(x)$ is assumed to be 
\begin{equation}\label{meas-symm}
\text{measurable and symmetric, }
\end{equation}
and uniformly elliptic, i.e., there exists positive constants $\lambda $ and $\Lambda$ such that 
\begin{equation}\label{eq-Ellipticity}
\lambda|\xi|^{2} \leq\langle\mathbb{A}(x)\xi, \xi\rangle \leq \Lambda|\xi|^{2} \quad \text{for all $x\in \Omega$}. 
\end{equation}
where $\langle \cdot, \cdot\rangle $ represents the usual dot product in $\RR^n$. 
Our main result (see Theorem \ref{maintheorem} below) states  that given any $q\in(1, \infty)$ and any weight $w \in A_{q}$, the Muckenhoupt class,  under some additional conditions on the matrix of coefficients and on the boundary, corresponding to ${\bf f}\in L^{q}_{w}(\Omega;\mathbb{R}^{n})$, there is a unique weak solution $u$ to \eqref{basic-pde} accompanied by the estimate 
\begin{equation}\label{weightedlp-lp}
\int_{\Omega}|\nabla u|^{q}w dx \leq C \int_{\Omega} |{\bf f}|^{q} w dx.  
\end{equation}
By a weak solution to \eqref{basic-pde} we mean any function $u\in W^{1,p}_{0}(\Omega)$ \emph{for some $p>1$} such that 
 \begin{equation}\label{def-weak-solution}
 \int_{\Omega} \langle \mathbb{A}(x) \nabla u(x),  \nabla \psi(x)\rangle \ dx = \int_{\Omega} \langle {\bf f}(x), \nabla \psi (x)\rangle\ dx \ \quad \text{for all}\  \psi \in C_c^{\infty}(\Omega).
 \end{equation}
Note that \eqref{def-weak-solution} makes sense as long as $\nabla u$ and ${\bf f} \in L^1_{loc}(\Omega; \mathbb{R}^{n})$. However, 
$W^{1,1}_{0}(\Omega)$ solutions to  \eqref{basic-pde} are generally not unique even when $\mathbb{A}$ has continuous coefficients and $\Omega$ is a ball (see \cite{BRE,Jin-Mazya,Serrin}). Thus in this paper we shall only adopt  $W^{1,p}_{0}(\Omega)$ solutions for some $p>1$. 

The emphasis of the paper is not on the existence of a solution, but rather in obtaining the tighter estimate \eqref{weightedlp-lp}. In fact, for ${\bf f}\in L^q_{w}(\Omega; \mathbb{R}^{n})$ with $w\in A_q$, $q>1$, it follows from the theory of $A_q$ weights that
${\bf f}\in  L^p(\Omega; \mathbb{R}^{n})$ for some $p>1$ (see Corollary \ref{Lebesgue-vs-weighted} below). Thus under small mean oscillation condition on $\mathbb{A}$, which we will make precise in the next section,  and   small Lipschitz constant condition on the boundary of $\Omega$, a unique $W^{1,p}_{0}(\Omega)$ solution $u$ exists satisfying the continuity estimate 
\begin{equation}\label{lp-lp}
\|\nabla u\|_{L^{p}(\Omega)} \leq C\, \|{\bf f}\|_{L^{p}(\Omega)}, \qquad p >1. 
\end{equation}
The well posedness of \eqref{basic-pde} for data  ${\bf f}\in  L^p(\Omega; \mathbb{R}^{n})$, $p>1$,  together with the $W^{1,p}$ estimate \eqref{lp-lp} was obtained earlier in  \cite{B1, BW}. 
In relation to this result, the main result in this paper states that under the same conditions the solution-gradient operator ${\bf f}\mapsto \nabla u$ is continuous not only on the bigger space $L^{p}(\Omega)$, as given by the estimate  \eqref{lp-lp},  but also on the smaller space $L^{q}_{w}(\Omega)$ as in \eqref{weightedlp-lp}. 
Earlier works on the (unweighted) bound \eqref{lp-lp}  include \cite{Agmon,  DiF, AQ} for smoother coefficients and domains, and  \cite{JK} for general Lipschitz domains but only for a restricted range of $p$ around $2$.   

When the matrix $\mathbb{A}=\mathbb{I}$ and $\Omega=\mathbb{R}^{n}$, by means of Fourier transform we see that gradients of solutions to \eqref{basic-pde} can be written as
$$\nabla u=-[{\bf R}\otimes {\bf R}]\, {\bf f},
$$   
where ${\bf R}=({R}_1, {R}_2, \dots, {R}_n)$ is the Riesz transform. Explicitly, this means that 
$$\nabla u(x)=-\frac{\Gamma(\frac{n+2}{2})}{\pi^{\frac{n}{2}}}\, {\rm p.v.} \int_{\mathbb{R}^{n}} \frac{(x-y)\otimes (x-y)}{|x-y|^{n+2}}{\bf f}(y)\, dy, \qquad x\in \mathbb{R}^{n}.$$
Once we  have the above representation, the theory of Calder\'on-Zygmund (see \cite{St1}) immediately gives estimate \eqref{lp-lp} for all $p>1$. Moreover, the results of Coifman and Fefferman \cite{CF}  on weighted norm inequalities for singular integrals  (see also the forerunner Hunt-Muckenhoupt-Wheeden \cite{HMW} in the one dimensional case)  yield weighted estimate \eqref{weightedlp-lp} for all weights $w\in A_q$ and all $q>1$. Note that in this scenario, the requirement that $w\in A_q$ is optimal (see \cite[Section V.4.6]{St2}).

By now it is well understood that  general continuity estimates of the type  \eqref{lp-lp} fails to hold in the event that either $\mathbb{A}$ has components with large jump discontinuity or the boundary of $\Omega$ is not smooth enough. The examples provided by Meyers  \cite{Mey} that involve highly oscillatory coefficients, and by Jerison and Kenig  \cite{JK} posed over domains with large Lipschitz constant have demonstrated  that gradient of solutions corresponding to smooth ${\bf f}$ may not belong to $L^{p}$ for 
large $p$.  These examples and a duality argument  justify the necessity of requiring slowly changing coefficients and domains with flatter boundary if one wants to obtain well posedness of \eqref{basic-pde} together with the continuity estimate \eqref{lp-lp} for data in $L^{p}(\Omega)$ for any  $p>1$.

Weighted $W^{1,q}$ estimates for equation \eqref{basic-pde} over a bounded, possibly, nonsmooth domain $\Omega$ 
have also been  considered in several recent papers.  The work \cite{Ph2} (see also \cite{M-P1, M-P2, M-P3}) yields 
weighted estimate \eqref{weightedlp-lp} for $q>2$ and for weights $w\in A_{q/2}$. In \cite{AP2}, the  first and last named authors  worked out the end-point case $q=2$ for weights $w\in A_1$. Those weighted estimates also hold true for equations with general nonlinear structures such as those that  are modeled after  the $p$-Laplacian. However,
as seen from the basic linear case $\mathbb{A}=\mathbb{I}$ and $\Omega=\mathbb{R}^{n}$, they are by no means sharp. For example, when $q=2$   one should expect that \eqref{weightedlp-lp}  holds for all weights $w\in A_2$ instead of $A_1$.  Likewise, for each $q>1$ one should expect that \eqref{weightedlp-lp}  holds for all $w\in A_q$. And this is precisely what we have  achieved in this paper  at least for a wide class of  linear equations potentially with discontinuous coefficients over  nonsmooth domains. Thus our result improves a similar result in \cite[Theorem 2.5]{BDS2016}, which treats the case of continuous coefficients over $C^1$ domains.  Moreover, it also complements the recent work  \cite{Kim-Dong2016} in which a similar problem is studied for equations that  involve a linear lower order term of the form:
\begin{equation*}
\left\{ \begin{array}{cl}
\text{div} \mathbb{A}(x) \nabla u +\lambda\, u=  \text{div}\, {\bf f} & \text{in}\, \Omega, \\
u = 0 & \text{on} \,\partial \Omega, 
\end{array}
\right.
\end{equation*}
where $\lambda$ is sufficiently large. The linear term  $\lambda\, u$ {\it for  large $\lambda$} is used in an essential way in \cite{Kim-Dong2016} to obtain \eqref{weightedlp-lp}, whereas we  obtain essentially the same result 
in the most natural case $\lambda=0$.

Weighted estimate of the form \eqref{weightedlp-lp} that is valid for  all weights  $w \in A_q$ is clearly a generalization of 
\eqref{lp-lp}. In fact, much more can be said about \eqref{weightedlp-lp}. It implies not only   inequality \eqref{lp-lp}
 but also its \emph{vector-valued} analogues (see \cite[Theorem 9.5.10]{Gra} and \cite{RdF}). Moreover, as a consequence of
\eqref{weightedlp-lp}, we can deduce sharp estimates in non-interpolating spaces such as Lorentz-Morrey spaces. 

Besides the global weighted estimate \eqref{weightedlp-lp}, we also obtain  local versions both in the interior and near the boundary. The local interior a priori estimate essentially states that, for any $q>1$ and  $w\in A_q$, under a small mean oscillation condition on $\mathbb{A}$,  it holds that 
\begin{equation}\label{local-int}
\int_{B_{d}(x_{0})} |\nabla u|^{q} w\  dx \leq C \int_{B_{2d}(x_{0})}(|{\bf f}|^{q} + \left|u/d\right|^{q})w \ dx 
\end{equation}
for all  $d\leq 1/M$ and $x_{0}\in \Omega$ such that $B_{Md}(x_{0})\subset \Omega$.  Here the constants $C>0$, $M>2$ depend on 
the $A_q$ constant of $w$ but are independent of $d, x_0$, and $u$. See Theorem \ref{main-interior-local-estimate}.
Inequality \eqref{local-int} can be viewed as an $L^q$ weighted Caccioppoli estimates that is well-known in the case $w\equiv 1$ and $q=2$. For $w\equiv 1$ and $q>2$, this has been obtained in \cite{B1, BW}. On the other hand, it appears to us that  in the sub-quadratic case $q\in (1,2)$, estimate  \eqref{local-int} is new even for $w\equiv 1$. A usual way to achieve this sub-quadratic Caccioppoli type estimate 
 is to localize the problem by  multiplying the solution $u$ itself by an appropriate  cut-off function $\varphi$ compactly supported in $B_{2d}(x_0)$  and then applying the global bound \eqref{lp-lp} over this ball. However, since the product $u\varphi$ then solves
\[
\text{div} [\mathbb{A}(x) \nabla (u \varphi)] = \text{div} [{\bf f} \varphi + \mathbb{A}(x) \nabla \varphi u] + \langle(\mathbb{A}(x)\nabla u - {\bf f}), \nabla \varphi\rangle, 
\]
which involves $\nabla u$  on the right-hand side, it is thus by no means obvious that one can absorb its $L^q$ norm  to the left-hand side. We shall show that this is possible by an integration by parts  combined with a covering/iteration argument. Local weighted a priori estimates similar to \eqref{local-int} but near the boundary of a Lipschitz domain with small Lipschitz constant can be found in    Theorem \ref{Lip-boundary-main}. These local interior and boundary estimates yield the global bound 
\eqref{weightedlp-lp} via a standard  covering argument combined with a weighted Sobolev embedding theorem.

\subsection*{Discussion on the approach}
The question of finding an optimal condition on the coefficients  and on the boundary so that  \eqref{basic-pde} is well-posed in a variety of function spaces and accompanied by a continuity estimate has garnered a lot of attention. Intuitively, for  reasonably good coefficients and domains,  if a solution $u$ to \eqref{basic-pde} exists then one expects the data ${\bf f}$ and $\nabla u$ belong to the same space, of course with the exception of   extreme and end-point spaces. There are a number of approaches to address this issue for elliptic boundary value problems in divergence form. One popular approach 
 is the approximation method pioneered by Caffarelli and Peral in \cite{CP} that avoids the use of singular integral theory directly but rather studies the integrability of gradient of solutions as a function of the deviation of the coefficients from constant coefficients. That method has been successfully implemented in \cite{B1, BW} with the use of 
Hardy-Littlewood maximal function to treat divergence form data on irregular domains; see also   \cite{AP1, AP2, M-P1, M-P2, M-P3, Jia-Li-Wang07, Jia-Li-Wang11, BYZ} where the method is used in different  function spaces  such as weighted Lebesgue and Lorentz spaces, Lorentz-Morrey spaces, and  Orlicz spaces. Another approach that employs a local version of the sharp maximal function of Fefferman and Stein was implemented in \cite{KZ1, KZ2} (see also the earlier work of T. Iwaniec \cite{I}). Unlike the approach in \cite{CP, B1, BW}, this method relies on the availability  of $C^{1, \alpha}$ regularity of  the associated homogenous equation with constant coefficients. This approach has been used to obtained a weighted estimate in \cite{Ph2}
(see also \cite{Kim-Dong2016}) in which a local weighted control by the sharp maximal function  was also obtained (see \cite[Corollary 2.7]{Ph2} or Lemma \ref{lemma_sharp_phuc} below).

In this paper, we shall follow the path of \cite{KZ1, KZ2, Ph2} to obtain \eqref{local-int} and its boundary analogue.
However, unlike the scenario in \cite{Ph2}, here we are subject to working with the largest possible class of weights for the weighted estimate under consideration. This forces us to come up with some new \emph{local comparison estimates} in $L^p$ spaces  for arbitrarily small $p>1$; see Lemma \ref{comparison-estimate} and Lemma \ref{flat-comparison-estimate-lemma}. It is fair to say that most of the technical parts of the paper lie in the proofs  of these  comparison estimates.

\subsection*{Organization of the paper} In Section \ref{MainStat} we will introduce  notations and state the main results of the paper. In Section \ref{Norm-ineq} we present some backgrounds on weighted norm inequalities that will be needed throughout the paper. To help us break down the proof of the main result, we reduce the main problem to certain {\it a priori} estimates of Caccioppoli type in Section \ref{a priori}. In Sections \ref{local interior} and \ref{local boundary}, we  prove  local  weighted $L^{q}$ interior and boundary estimates, respectively. Finally, we put the local estimates together in  Section \ref{global} to obtain the global results stated in Section \ref{MainStat}.

 \section{Notations and Statement of main results}\label{MainStat}
 \subsection{Notations}
 The notation $\langle \cdot, \cdot\rangle $ represents the usual dot product in $\RR^n$, and when we find it convinient we also use $\cdot$ notation for it. 
 
 The function spaces $L^{p}(\Omega)$ and  $W^{1,p}(\Omega)$ are the usual Lebesgue and Sobolev spaces, with  $L^{p}_{loc}(\Omega)$,  $W^{1,p}_{loc}(\Omega)$ being their local versions. The space $W^{1,p}_{0}(\Omega)$ is defined as the 
completion of $C_c^\infty(\Omega)$ under the norm of $W^{1,p}(\Omega)$. For $r\in [1, \infty]$, we write $r'=r/(r-1)$ to denote its H\"older conjugate exponent.

A nonnegative function $w \in L^1_{loc}(\RR^n)$ is said to belong the Muckenhoupt weight class $A_q$ if 
   \begin{equation*}
  \begin{split}
    [w]_{A_q} &:= \sup_{B \subset \RR^n} \left(  \fint_B w(x) \ dx \right) \left(  \fint_B w(x)^{\frac{-1}{q-1}} \ dx \right)^{q-1} < +\infty, \quad 1 < q < \infty,\\
    [w]_{A_1} &:= \sup_{B \subset \RR^n} \left(  \fint_B w(x) \ dx \right) \|w^{-1}\|_{L^{\infty}(B)} < +\infty, \quad q=1,
    \end{split}
   \end{equation*}
   where the supremum is taken over all balls $B \subset \RR^n$,  the quantity $[w]_{A_q}$ is called the $A_q$ constant of $w$, and 
   for function $g$ and any set $D$, we have used the notation   
 \[\langle g\rangle_{D} = \fint_{D} g(y)dy = \frac{1}{|D|}\int_{D} g(y)dy.  \]
  
	For a given weight $w$,  the weighted Lebesgue space $L^{q}_{w}(\Omega)$ is the space of measurable functions $f$ such that $\|f\|_{L^{q}_{w}(\Omega)}^{q} := \int_{\Omega}|f|^{q} w(x)dx <+\infty$.  Note that for $1< q<+\infty$ the space 
	$C_c^\infty(\Omega)$ is dense in the  weighted space $L^{q}_{w}(\Omega)$ provided $w\in A_{q}$.
	 The weighted Sobolev space $W^{1, q}_{w}(\Omega)$ is the set of measurable functions  $u\in L^{q}_{w}(\Omega)$ whose weak gradient $\nabla u\in L^{q}_{w}(\Omega)$. The space $W^{1, q}_{w}(\Omega)$ is a Banach space with the norm  
 \[
 \|u\|_{W^{1, q}_{w}(\Omega)} ^{q} = \|u\|_{L^{q}_{w}}^{q} + \|\nabla u\|_{L^{q}_{w}}^{q}. 
 \]
 The closure of $C_{c}^{\infty}(\Omega)$ with respect to the norm $\|\cdot \|_{W^{1, q}_{w}(\Omega)}$ is denoted by $W^{1, q}_{w, 0}(\Omega)$. 
 We write $u\in W^{1, q}_{w, loc}(\Omega)$ if for any $\Omega'\Subset \Omega$, we have $u\in W^{1,q}_{w}(\Omega')$. For functions $u\in W^{1,q}_{w, 0}(\Omega)$ with $w\in A_q$, we shall use the following version of weighted Sobolev inequality that can be deduced from \cite[Theorem 15.23]{HKM}:  
\begin{equation}\label{weightedSob}
\left(\frac{1}{w(B)}\int_{\Omega} |u|^{\frac{nq}{n-1}} w dx\right)^{\frac{n-1}{nq}}\leq C(n, q, [w]_{A_q}) |B|^{\frac{1}{n}} \left(\frac{1}{w(B)}\int_{\Omega} |\nabla u|^q w dx\right)^{\frac{1}{q}}
\end{equation}
for any ball $B$ that contains $\Omega$. 

As we stated earlier, our aim is to prove that if $w\in A_{q}$ and ${\bf f}\in L^{q}_{w}(\Omega)$, then $\nabla u\in L^{q}_{w}(\Omega)$. 
 To achieve this, we need to require that the coefficient matrix $\mathbb{A}$ has a small mean oscillation and the boundary is flat enough,  conditions that we will make precise next.  
\begin{defn}
For given $\mathfrak{K}>0$ and $\delta>0$, we say $\mathbb{A}$ is $(\delta,\mathfrak{K})$-BMO if 
\[
\|\mathbb{A}\|_{*, \mathfrak{K}} := \sup_{x\in \overline{\Omega}} \ \sup_{0<\rho <\mathfrak{K}} \fint_{B(x,\rho)\cap \Omega}|\mathbb{A} - \langle\mathbb{A}\rangle_{B(x, \rho)\cap \Omega} |dx < \delta.
\]
 \end{defn}
 We will work on Lipschitz domains with small Lipschitz constant as defined below. 
 \begin{defn}
 For a given $\mathfrak{K}$ and $\delta> 0$, we say that $\Omega$ is $(\delta,\mathfrak{K})$-Lip if for every $x_{0}\in \partial\Omega$, there exists a Lipschitz continuous function $\Gamma: \mathbb{R}^{n-1}\to \mathbb{R}$ such that $\|\nabla \Gamma\|_{L^{\infty}(\RR^{n-1})} <\delta$ and, upon rotating and   relabeling of coordinates if necessary, 
 \[
 \Omega\cap B_{\mathfrak{K}}(x_{0}) = \{x=(x', x_{n})\in B_{\mathfrak{K}}(x_{0}): x_{n} > \Gamma(x')\}.
 \]  
 \end{defn}
 
 \subsection{Statements of main results}
 The main result we will be proving is stated the following theorem. 
 \begin{theorem}\label{maintheorem}
Suppose that $\mathfrak{K}>0$, $M_{0} > 0$,  $1 < q < \infty$ and $w\in A_{q}$ such that $[w]_{A_{q}} \leq M_{0}$. Suppose also that $\mathbb{A}$ satisfies \eqref{meas-symm} and \eqref{eq-Ellipticity}. Then there exists $\delta=\delta(\lambda,
\Lambda, q, n,  M_{0})> 0$ such that for any $\mathbb{A}$ that is $(\delta, \mathfrak{K})$-BMO and $\Omega$ a $(\delta, \mathfrak{K})$-Lip domain, and any  ${\bf f}\in L^{q}_{w}(\Omega;\mathbb{R}^{n})$, there exists a unique weak solution $u$ to \eqref{basic-pde} where $u \in W_0^{1,p} (\Omega)$ for some $p>1$, 
 $\nabla u\in L^{q}_{w}(\Omega;\mathbb{R}^{n})$  and satisfies the continuity estimate   
\begin{equation}\label{stability}
\int_{\Omega} |\nabla u|^q w(x) \ dx \leq C \int_{\Omega} |{\bf f}|^{q}w(x)dx. 
\end{equation}
The constant $C$  depends only on $\lambda,\Lambda,q,n, M_{0}$ and $\text{diam}(\Omega)/\mathfrak{K}$. 
\end{theorem}

Before listing the consequences of Theorem  \ref{maintheorem}, we mention that by the theory of extrapolation
of Rubio de Francia, it is sufficient to prove  it for the case $q=2$ (see, e.g., \cite[Appendix A]{AP2} and \cite{GR}). However, Theorem \ref{maintheorem} will be proved directly for all $q>1$ and thus the  
theory of extrapolation is not needed in this paper.

There are a number of  corollaries that can be deduced from the above weighted $L^{q}$ estimate. The first involves refining estimate \eqref{stability}
when the data is in the weighted Lorentz space. As we will see in the proof, the weighted Lorentz space estimate given below is possible via interpolation using the properties of $A_{q}$ weights and the fact that the solution-gradient operator ${\bf f}\mapsto \nabla u$ is  linear. For given $0<r\leq\infty$ and $0<q< \infty$,  and a weight function $w$,  the weighted  Lorentz space $L^{q, r}_{w}(\Omega)$  is defined as the space of measurable functions $g$  on $\Omega$ such that 
\[
\|g\|_{L^{q, r}_{w}(\Omega)}:=\left(q \int_{0}^{\infty}( t^{q}w(\{ x\in \Omega: |g(x)| > t \}) )^{r/q} \frac{d t}{t}\right)^{1/r} < \infty,
\]
for $r \neq \infty$; and for $r  = \infty$, the space $L_{w}^{q,\infty}(\Omega) $ is the usual weighted weak-$L^{q}$ space with quasinorm 
\[
\|g\|_{L^{q, \infty}_{w}(\Omega)}:= \sup_{t > 0} t w(\{x\in \Omega: |g(x)| > t \})^{1/q}. 
\]
When $q = r$, the space $L^{q, r}_{w}(\Omega)$ is precisely the weighted $L^{q}$ space and will be denoted simply by $L^{q}_{w}(\Omega)$. 
\begin{corollary}\label{weightedLorentz}
Suppose that $M_{0} > 0$ and  $w\in A_{q}$ such that $[w]_{A_{q}} \leq M_{0}$. 
 Suppose also that $\mathfrak{K}>0$, $0< r \leq \infty$, and $1 < q < \infty$ are given, and that $\mathbb{A}$ satisfies \eqref{meas-symm} and\eqref{eq-Ellipticity}. Then there exists $\delta=\delta(\lambda, \Lambda, q, n,  M_{0})> 0$ such that for any $\mathbb{A}$ that is $(\delta, \mathfrak{K})$-BMO and $\Omega$ a $(\delta, \mathfrak{K})$-Lip domain, and any  ${\bf f}\in L^{q, r}_{w}(\Omega;\mathbb{R}^{n})$, there exists a unique weak solution $u$ to \eqref{basic-pde} where $u \in W_0^{1,p} (\Omega)$ for some $p>1$, 
 $\nabla u\in L^{q, r}_{w}(\Omega;\mathbb{R}^{n})$  and satisfies the continuity estimate   
\begin{equation*}
\|\nabla u\|_{L^{q, r}_{w} (\Omega)} \leq C \| {\bf f}\|_{L^{q, r}_{w}(\Omega)}. 
\end{equation*}
 The constant $C$  depends only on $ \lambda, \Lambda, q,n, r, M_{0}$ and $\text{diam}(\Omega)/\mathfrak{K}$. 
\end{corollary}

Another corollary of the weighted $L^{q}$ estimate is an estimate in (unweighted) Lorentz-Morrey spaces, which cannot be deduced by the usual means of interpolation.   
Given $0 < r \leq \infty$, $0 < q < \infty$ and $\theta\in (0, n]$ the Lorentz-Morrey space $\mathcal{L}^{q, r;\theta}(\Omega)$ is the set of measurable functions $g$ such that 
\[
\| g\|_{\mathcal{L}^{q, r;\theta}(\Omega)} := \sup_{\stackrel{0< \rho \leq \text{diam}(\Omega)}{z\in \Omega}} \rho^{\frac{\theta- n}{q} } \| g\|_{L^{q, r}(B_{\rho}(z)\cap \Omega)}  < +\infty,
\]
where the norm $\|\cdot\|_{L^{q, r}}$ is the Lorentz quasinorm corresponding to the Lebesgue measure.  When $\theta = n$, the space  $\mathcal{L}^{q, r;\theta}(\Omega)$ is the standard Lorentz space and is denoted by  ${L}^{q, r}(\Omega)$. When $q = r$,  the space $\mathcal{L}^{q, r;\theta}(\Omega)$ is the  usual Morrey space and is denoted by  $\mathcal{L}^{q;\theta}(\Omega)$. 
\begin{theorem}\label{Lorentz-Morrey}
 Suppose that $\mathfrak{K}>0$, $0< r \leq \infty$, $1 < q < \infty$ and $0 < \theta\leq n$. Suppose also that $\mathbb{A}$ satisfies \eqref{meas-symm} and \eqref{eq-Ellipticity}. Then there exists $\delta=\delta(\lambda,
\Lambda, q, n,  \theta)> 0$ such that for any $\mathbb{A}$ that is $(\delta, \mathfrak{K})$-BMO and $\Omega$ a $(\delta, \mathfrak{K})$-Lip domain, and any  ${\bf f}\in \mathcal{L}^{q, r; \theta}(\Omega;\mathbb{R}^{n})$, there exists a unique weak solution $u$ to \eqref{basic-pde} such that  
 $\nabla u\in \mathcal{L}^{q, r;\theta}(\Omega;\mathbb{R}^{n})$  and satisfies the continuity estimate   
\begin{equation*}
\|\nabla u\|_{\mathcal{L}^{q, r; \theta}(\Omega)} \leq C \| {\bf f}\|_{\mathcal{L}^{q, r;\theta}(\Omega)}. 
\end{equation*}
The constant $C$   depends only on $\lambda,\Lambda, n, q, r, \theta$ and $\text{diam}(\Omega)/\mathfrak{K}$. 
\end{theorem}
%
 \section{Preliminaries on Weighted Norm inequalities}\label{Norm-ineq}
  In this section we collect relevant norm inequalities related to the Muckenhoupt class $A_q$.  
  We begin with the very important property of Muckenhoupt weights, which  is the {\it reverse H\"{o}lder property} (see \cite[Theorem 9.2.2]{Gra}). 
  \begin{lemma}
  \label{reverse_holder}
    Let $M_{0} >0$, $1 <q<\infty$ and $w \in A_q$ such that $[w]_{A_{q}} \leq M_{0}$. Then 
    there exists constants $C = C(n,q, M_{0})>0$ and $\gamma = \gamma(n,q, M_{0})>0$ such that for every ball $B$, we have
   $$ \left( \frac{1}{|B|} \int_B w(x)^{1+\gamma} \ dx \right)^{\frac{1}{1+\gamma}} \leq \frac{C}{|B|} \int_B w(x) \ dx .$$
  \end{lemma}
  Note in particular that $w\in L^{1 + \gamma}_{loc}(\mathbb{R}^{n})$. 
  A nontrivial result that follows from the  reverse H\"{o}lder property of Muckenhoupt weights is the  ``open-ended property''  (see e.g, \cite[Corollary 9.2.6]{Gra})  and will be very useful in proving our main theorem. 
  \begin{lemma}
  \label{open_ended}
   Let $M_{0} >0$, $1 <q<\infty$ and $w \in A_q$ such that $[w]_{A_{q}} \leq M_{0}$. Then there exist a constant $q_0  =q_0 (n,q, M_{0})\in (1, q)$ and a constant $C = C(n, q, M_{0})$ such that $w \in A_{q_0}$ satisfying the estimate $$[w]_{A_{q_0}} \leq C .$$
   In particular, this says that $A_q = \bigcup_{1\leq q_0<q} A_{q_0}$. 
  \end{lemma}
  A simple H\"older's inequality yields the following useful result. 
  \newcommand{\bart}{\overline{\mathfrak{t}}}
  \newcommand{\tbar}{\underline{\mathfrak{t}}}
  \begin{corollary}\label{Lebesgue-vs-weighted}
  Let $w\in A_{q}$, $q>1$. Denote $\overline{\mathfrak{t}}= q\bigg(1 + \frac{1}{\gamma}\bigg) $ and $p= \frac{q}{q_{0}} > 1$ where $\gamma >0$ and $q_{0} > 1$ are numbers guaranteed by Lemma \ref{reverse_holder} and Lemma \ref{open_ended}, respectively. Then we have that $L^{\bart}(\Omega) \subset L^{q}_{w}(\Omega) \subset L^{p}(\Omega)
$  and 
  \[
  W^{1,\bart }_{0}(\Omega) \subset W^{1, q}_{w, 0}(\Omega) \subset W^{1, p}_{0}(\Omega). 
  \]
  \end{corollary}
  Weights in the $A_{q}$ class are intimately related to the {\it Hardy-Littlewood maximal function} which is defined for  a function $f \in L^1_{loc}(\RR^n)$,
   as  
 \begin{equation*}
  \mm f(x) = \sup_{r>0} \frac{1}{|B_r(x)|}\int_{B_r(x)} |f(y)| \ dy.  
 \end{equation*}
The well known result on the necessary and sufficient condition for the boundedness of the maximal function on weighted $L^{p}$ spaces, \cite{M, Gra}, is now stated.  
  \begin{lemma}
   \label{weighted_maximal}
   Let $M_{0} >0$, $1 <q<\infty$ and $w \in A_q$ such that $[w]_{A_{q}} \leq M_{0}$. 
   Then there exists a constant $C=  C(n,q,M_{0})$ such that 
   \begin{equation}
    \label{maximal_estimate_weight}
    \| \mm f \|_{L^q_w(\RR^n)} \leq C \| f \|_{L^q_w(\RR^n)}
   \end{equation}
   for all $f \in L^q_w(\RR^n)$. Conversely if \eqref{maximal_estimate_weight} holds for all $f \in L^q_w(\RR^n)$, then necessarily $w \in A_q$. 
  \end{lemma}
We will also need a truncated version of the {\it Fefferman-Stein sharp maximal function} that is defined for each $\rho>0$ by 
\begin{equation*}
  \mm^{\#}_{\rho} f(x) = \sup_{r\in (0, \rho]} \frac{1}{|B_r(x)|}\int_{B_r(x)} |f(y) - \avg{f}{B_r(x)}| \ dy. 
 \end{equation*}
 It is easy to see from these definitions and Lebesgue differentiation theorem that 
  \begin{equation*}
   f \leq  \mm f \qquad \text{and}\qquad \mm_{\rho}^{\#} f \leq  2 \ \mm f. 
  \end{equation*}

Our use of $\mm_{\rho}^{\#}$ lies in the following  key estimate that bounds the weighted $L^{q}$ norm of a function by the weighted norm of its truncated sharp maximal function. This lemma was obtained in \cite[Corollary 2.7]{Ph2}; see also 
\cite{KZ1} and \cite{FS} for the unweighted case. 
  \begin{lemma}[\cite{Ph2}]
  \label{lemma_sharp_phuc}
  Let $M_{0} >0$, $1 <q<\infty$ and $w \in A_q$ such that $[w]_{A_{q}} \leq M_{0}$.   Then there exist a constant $\kappa = \kappa(n,q,M_{0})>\sqrt{n}$ and $C = C(n,q, M_{0})>0$ such that for $f \in L^q(\RR^n)$ or $f\in L^q_{w}(\RR^n)$ with $\textrm{spt}(f) \subset B_\rho(x_0)$ for some $\rho>0$, we have the estimate
   \begin{equation*}
    \int_{B_\rho(x_0)} |f(x)|^q w(x) \ dx \leq C\int_{B_{\kappa \rho}(x_0)} [\mm^{\#}_{\kappa \rho}f(x)]^q w(x) \ dx.
   \end{equation*}
  \end{lemma}
  
\section{Proof of the main result based on {\em a priori} estimates of Caccioppoli type} \label{a priori}
We begin by reiterating the point that given ${\bf f}\in L^{q}_{w}(\Omega;\mathbb{R}^{n})$, for $1 < q < \infty$, there is a corresponding unique solution to \eqref{basic-pde}. Indeed, using $p>1$ in Corollary \ref{Lebesgue-vs-weighted} 
a simple application of H\"older's inequality shows that  ${\bf f} \in L^{p}(\Omega;\mathbb{R}^{n})$  with the estimate
\begin{equation}\label{holder_weights}
\left(\fint_{B}|{\bf f}|^{p}\ dx\right)^{1/p }\\
 \leq C([w]_{A_q})   \left(\frac{1}{w(B)}\int_{B}|{\bf f}|^{q} w \ dx \right)^{1/q},
\end{equation}
where $B$ is any ball that contains $\Omega$ and ${\bf f}$ is set to be zero outside $\Omega$.
Next we apply  \cite[Theorem 1.5]{B1} to find constants $C$ and $\delta$, and a unique weak solution $u \in W^{1, p}_{0}(\Omega)$ solving  \eqref{basic-pde} corresponding to ${\bf f}$ that satisfies the estimate 
\begin{equation}
 \label{byun_a_priori}
 \|\nabla u\|_{L^{p}(\Omega)} \leq C \|{\bf f}\|_{L^{p}(\Omega)}, 
\end{equation}
provided $\mathbb{A}$ is $(\delta, \mathfrak{K})$-BMO and $\Omega$ is $(\delta, \mathfrak{K})$-Lip. So when we say a solution to \eqref{basic-pde} corresponding to ${\bf f}$, we are referring to this solution by uniqueness. In passing, we note that even though it is not clearly stated in \cite[Theorem 1.5]{B1},  the dependence of the constant $C$ on the domain $\Omega$ is only through the ratio $\text{diam}(\Omega)/\mathfrak{K}$. For this we refer to the recent paper  \cite[Theorem 1.8]{M-P3} when applied to the linear problem.  

We will prove Theorem \ref{maintheorem} as a consequence of the following a priori estimates of Caccioppoli type estimates. 
\begin{theorem}\label{maintheorem-Cacciaopoli}
Suppose that  $\mathfrak{K}>0,$ $M_{0} >0$, $1 <q<\infty$ and $w \in A_q$ such that $[w]_{A_{q}} \leq M_{0}$. Suppose also that  $\mathbb{A}$ satisfies \eqref{meas-symm} and \eqref{eq-Ellipticity}. Then there exists  $\delta=\delta(\lambda, \Lambda, q, n,  M_{0})> 0$ such that for any $\mathbb{A}$ that is $(\delta, \mathfrak{K})$-BMO and $\Omega$ a $(\delta, \mathfrak{K})$-Lip domain, if ${\bf f}\in C_{c}^{\infty}(\Omega;\mathbb{R}^{n})$ then the corresponding weak  solution $u$ to \eqref{basic-pde}  satisfies 
\begin{equation}\label{stability-Cacciaopoli}
\int_{\Omega} |\nabla u|^q w(x) \ dx \leq C \int_{\Omega} \left(|{\bf f}|^{q} + |u/\mathfrak{K}|^{q}\right)w(x)\ dx. 
\end{equation}
The constant $C$ depends only on $\lambda, \Lambda, q, n, M_{0}$ and $\text{diam}(\Omega)/\mathfrak{K}$. 
\end{theorem}
We will postpone the proof of Theorem \ref{maintheorem-Cacciaopoli} for later sections. For now we will assume its validity to prove the main result of the paper.  
\newline

\noindent\begin{proof}[Proof of Theorem \ref{maintheorem}]
Suppose that  vector field  ${\bf f}\in L^{q}_{w}(\Omega; \RR^n)$, and $u\in W^{1, p}_{0}(\Omega)$ is its corresponding solution.  We want to show that $u$ satisfies inequality  \eqref{stability}. Using the density of space of smooth functions in the weighted $L^{q}$ space, pick a sequence of vector fields  ${\bf f}_{m} \in C_{c}^{\infty}(\Omega;\mathbb{R}^{n})$ such that, 
\[
{\bf f}_{m} \to {\bf f} \quad \text{in}\ \   L^{q}_{w}(\Omega; \RR^n),  \qquad  \text{as} \ \ m\to \infty. 
\]
For each $m$, applying Theorem \ref{maintheorem-Cacciaopoli} and noting ${\bf f}_{m} \in L^{\infty}(\Omega;\mathbb{R}^{n})$, and $\Omega$ is a bounded set, we can choose $C$  and $\delta$ such that the corresponding solution $u_{m}$ to \eqref{basic-pde} belongs to $W^{1, \bart}_{0}(\Omega)\subset W^{1,q}_{w,0}(\Omega)$, where $\bart$ is as in Corollary \ref{Lebesgue-vs-weighted}, 
with the estimate 
\begin{equation}\label{stability-Cacciaopoli-m}
\int_{\Omega} |\nabla u_{m}|^q w(x) \ dx \leq C \int_{\Omega} \left(|{\bf f}_{m}|^{q} + |u_{m}/\mathfrak{K}|^{q}\right)w(x)\ dx.  
\end{equation}
We will demonstrate next that it is possible to absorb the term $\int_{\Omega} |u_{m}/\mathfrak{K}|^{q} w \, dx $ on the right hand side of \eqref{stability-Cacciaopoli-m} to obtain that, up on a new constant C independent of $m$, 
\begin{equation}\label{Cinfty-right-estimate}
\int_{\Omega} |\nabla u_{m}|^q w(x) \ dx \leq C \int_{\Omega} |{\bf f}_{m}|^{q} w(x)\ dx. 
\end{equation}
To that end, let $B$ be any ball of radius ${\rm diam}(\Omega)$  such that $\Omega \subset B$ and denote the weight $\overline{w} = \frac{w}{w(B)}$. Clearly, we have  $\overline{w}\in A_{q} $ and $[\overline{w}]_{A_{q}} = [w]_{A_{q}}$.
 Let $\epsilon >0$ be a small number that will be determined later and extend $u_{m}$ by zero outside $\Omega$, then by an interpolation inequality (see \cite[Proposition 1.1.14]{Gra}), we have, 
\begin{equation}
 \label{interpolation_one}
\left(\int_{B} |u_{m}/\mathfrak{K}|^{q} \  \overline{w} \ dx\right)^{1/q} \leq  \mathfrak{K}^{-1}\|u_{m}\|_{L^{\epsilon}_{\overline{w}}(B)}^{\theta}\| u_{m}\|_{L_{\overline{w}}^{\frac{nq}{n-1}}(B) }^{1-\theta}, \qquad \text{where}\  \theta = \frac{\epsilon} {nq - \epsilon(n-1)}. 
\end{equation}
Since we have $u_{m} \in W^{1,q}_{w,0}(\Omega)$, the zero extension of $u_{m}$ outside of $\Omega$ will be in $W^{1, q}_{w,0}(B).$  We can then apply the weighted Sobolev embedding \eqref{weightedSob} to the last term in \eqref{interpolation_one} followed by Young's inequality with exponents $\frac{1}{\theta}$ and $\frac{1}{1-\theta}$ to obtain
\begin{equation}\label{umathfrakk}
\begin{split}
\left(\int_{B} |u_{m}/\mathfrak{K}|^{q} \ \overline{w} \ dx\right)^{1/q} &\leq C\  \mathfrak{K}^{-1} \text{diam}(\Omega)^{1-\theta}\|u_{m}\|_{L^{\epsilon}_{\overline{w}}(B)}^{\theta}\| \nabla u_{m}\|_{L_{\overline{w}}^{q}(B)}^{1-\theta}\\
& \leq C \ \mathfrak{K}^{\frac{-1}{\theta}} \text{diam}(\Omega)^{\frac{1-\theta}{\theta}}\|u_{m}\|_{L^{\epsilon}_{\overline{w}}(B) } + \eta \|\nabla u_{m}\|_{L_{\overline{w}}^{q}(B)}
\end{split}
\end{equation}
for any $\eta > 0$.  We will eventually choose $\eta = \frac{1}{2}$.  

On the other hand,  if $ \gamma_{0} = 1+\gamma> 1$ where $\gamma >0$ is as in Lemma \ref{reverse_holder} from the reverse H\"older property  of $w \in A_{q}$, 
we then see that 
\[
\begin{split}
\int_{\Omega} |u_{m}|^{\epsilon} \ \overline{w} \ dx  
\leq \frac{1}{w(B)} \left( \int_{B} |u_{m}|^{ \frac{\epsilon \gamma_{0}}{\gamma_{0}-1}} dx \right)^{(\gamma_{0}-1)/\gamma_{0}} \left( \int_{B}w ^{\gamma_{0}}dx \right)^{1/\gamma_{0}}\leq C \left( \fint_{B} |u_{m}|^{ \frac{\epsilon \gamma_{0}}{\gamma_{0}-1}} dx \right)^{(\gamma_{0}-1)/\gamma_{0}}. 
\end{split}
\] 
We may now use the estimate \eqref{byun_a_priori}
to make an appropriate choice of $\epsilon$ as follows:  choose 
\[\epsilon \leq \left( \frac{np}{n-p} \right) \left(\frac{\gamma_0 -1}{\gamma_0}\right) \,\text{ if $p < n$, else any $ \epsilon>0$ if $p \geq n$.}
\]
 Then applying  { Sobolev embedding} and by this choice of $\epsilon$, we get 
\[
\begin{split}
\left( \fint_{B} |u_{m}|^{ \frac{\epsilon \gamma_0}{\gamma_0-1}}\  dx \right)^{(\gamma_{0}-1)/\gamma_{0}} &\leq C \text{diam}(\Omega)^{\epsilon} \left(\fint_{B}|\nabla u_{m}|^{p} \ dx \right)^{\epsilon/p} \leq  C  \text{diam}(\Omega)^{\epsilon} \left(\fint_{B}|{\bf f}|^{p}\ dx\right)^{\epsilon/p }\\
& \leq C  \text{diam}(\Omega)^{\epsilon}  \left(\int_{B}|{\bf f}|^{q} \overline{w} \ dx \right)^{\epsilon/q},
\end{split}
\]
where we used \eqref{byun_a_priori} and \eqref{holder_weights} in the second and last inequality, respectively. Thus we find that
\begin{equation}\label{uepsilonw}
\|u_{m}\|_{L^{\epsilon}_{\overline{w}}(B)}\leq C \text{diam}(\Omega)  \left(\int_{B}|{\bf f}|^{q} \overline{w} \ dx \right)^{1/q}.
\end{equation}

 Combining \eqref{umathfrakk} and \eqref{uepsilonw} together,  we obtain that 
\[
\left(\int_{B} |u_{m}/\mathfrak{K}|^{q} \ \overline{w} \ dx\right)^{1/q}   \leq C\, \mathfrak{K}^\frac{-1}{\theta} \text{diam}(\Omega)^{\frac{1}{\theta}} \left( \int_{B} |{\bf f }|^{q} \overline{w}\  dx \right)^{1/q}  + \eta  \left( \int_{B} |\nabla u_{m}|^{q} \overline{w}\ dx \right)^{1/q}.
\]
Using this inequality in \eqref{stability-Cacciaopoli-m}, by making the choice of  $\eta= \frac{1}{2}$, we can absorb the last term of \eqref{stability-Cacciaopoli-m} to its left hand side to obtain \eqref{Cinfty-right-estimate}. 

Finally, by linearity, for each $m$ we have that  $u_m - u$ uniquely solves \eqref{basic-pde} corresponding to the data ${\bf f}_{m} - {\bf f} \in L^{p}(\Omega)$. Therefore we have the convergence 
\[
\| \nabla u_{m} - \nabla u \|_{L^{p}} \leq C' \|{\bf f}_{m} - {\bf f}\|_{L^{p}} \leq C\|{\bf f}_{m} - {\bf f}\|_{L^{q}_{w}(\Omega)} \to 0,\quad  \text{as $m\to \infty$, }
\]
provided 
$\mathbb{A}$ is $(\delta, \mathfrak{K})$-BMO and $\Omega$ is $(\delta, \mathfrak{K})$-Lip. 
As a consequence, up to a subsequence, $\nabla u_{m} \to  \nabla u$ almost everywhere in $\Omega$ as $m\to \infty$. Now we apply Fatou's lemma to \eqref{Cinfty-right-estimate}, to conclude that
\[
\int_{\Omega} |\nabla u|^{q} w dx\leq \liminf_{m\to \infty} \int_{\Omega} |\nabla u_{m}|^{q} w dx\leq C \lim_{m\to \infty} \int_{\Omega} |{\bf f}_{m}|^{q}w dx =C \int_{\Omega} |{\bf f}|^{q}w dx, 
\]
as desired. 
\end{proof}

The rest of this paper will be devoted to proving Theorem \ref{maintheorem-Cacciaopoli}.  It turns out that to prove the estimate \eqref{stability-Cacciaopoli} what we need is the weaker assumptions ${\bf f}\in L^{q}_{w}(\Omega;\mathbb{R}^{n})$ and  the corresponding solution $u\in W^{1, q}_{w, 0}(\Omega)$, which hold true whenever  ${\bf f}\in L^{\infty}(\Omega;\mathbb{R}^{n})$.   In all what follows,  we assume only these weaker conditions.  We prove the theorem by obtaining first local  interior and boundary estimates for the solution which is done in the next two sections. The local estimates employ a comparison argument that compares the solution $u$ with a solution of a  homogeneous  equation with constant coefficients. This allows us to obtain mean oscillation estimates for $\nabla u$ that is used to estimate  the weighted sharp maximal function of $\nabla u$. Next, we use Lemma \ref{lemma_sharp_phuc} to obtain the desired weighted gradient estimate. Such an approach has been used for $p$-Laplacian type problems in \cite{Ph2,KZ1,KZ2} and recently for linear problems in \cite{Kim-Dong2016}.

\section{Local interior estimates}\label{local interior}
The main theorem we will be proving in this section is the following local interior estimate for gradients of solution.
\begin{theorem}\label{main-interior-local-estimate} 
Suppose that $\mathfrak{K}>0$, $M_{0} > 0$,  $1 <q<\infty$, and $w\in A_{q}$ such that $[w]_{A_{q}} \leq M_{0}$. Suppose also that 
$\mathbb{A}$ satisfies  \eqref{meas-symm} and \eqref{eq-Ellipticity},  ${\bf f}\in L^{q}_{w, loc}(\Omega;\mathbb{R}^{n})$ and  $u\in W^{1, q}_{w, loc}(\Omega)$ is a weak solution of 
\begin{equation}\label{interior-pde}
\text{div}\, \mathbb{A}(x)\nabla u = \text{div}\, {\bf f}(x)\quad \text{in}\,\,\Omega. 
\end{equation}
Then there exist  constants  $\delta_{0} >0$, $M > 2$ and $C>0$
such that, whenever $\mathbb{A}$ is $(\delta, \mathfrak{K})$-BMO with $\delta \leq \delta_{0}$, it holds that 
\begin{equation}\label{eq15}
\int_{B_{d}(x_{0})} |\nabla u|^{q} w(x)\  dx \leq C \int_{B_{2d}(x_{0})}\left(|{\bf f}|^{q} + \left|\frac{u}{d}\right|^{q}\right)w (x)\ dx 
\end{equation}
for all $d \leq \mathfrak{K}/M$ and $x_{0}\in \Omega$ with $B_{Md}(x_{0})\subset \Omega$.  The constants $\delta_0, M$, and $C$
depend only on $\lambda, \Lambda, n, q,$ and  $M_0$.
\end{theorem}

We remark that \eqref{eq15} is a weighted Caccioppoli type inequality, which in the case $1<q<2$  appears to be new even for  $w\equiv 1$. 

\begin{corollary}
Under the hypothesis of Theorem \ref{main-interior-local-estimate}, for every  $\Omega' \Subset \Omega$, there exists a constant $\delta_{0} = \delta_0 (\lambda, \Lambda, n, q, M_0)>0$,  
such that 
whenever $\mathbb{A}$ is $(\delta, \mathfrak{K})$-BMO with $\delta \leq \delta_{0}$, there exists $C = C(\lambda, \Lambda,n,q, M_{0}, \mathfrak{K}, \text{diam}(\Omega), \text{dist}(\Omega', \partial\Omega))>0$ such that 
\begin{equation*} 
\int_{\Omega'} |\nabla u|^{q} w(x) \ dx \leq C \int_{\Omega}\left(|{\bf f}|^{q} + \left|\frac{u}{d}\right|^{q}\right)w(x)\  dx. 
\end{equation*}
\end{corollary}

\subsection*{Local interior estimate  set up} The remaining part of this section is based on the following set up. We assume that ${\bf f}\in L^{q}_{w, loc}(\Omega; \mathbb{R}^{n})$, $u\in W^{1, q}_{w, loc}(\Omega)$ and $u$ weakly solves \eqref{interior-pde}. 
Taking $p$ as in Corollary \ref{Lebesgue-vs-weighted} we have ${\bf f}\in L^{p}_{loc}(\Omega)$ and $u \in W^{1, p}_{loc}(\Omega). $
For universal constants we use below, $C$ or $C_{0}$,  we suppress their dependence on $\lambda, \Lambda, n, q, M_{0}$. When necessary, we will specify the dependence of the constants on particular parameters to avoid confusion. 

 Fix an $x_{0}\in \Omega$ and let $\kappa > \sqrt{n}$ be as in Lemma \ref{lemma_sharp_phuc}. Let  $h\geq 2$ and  $d>0$ be constants to be chosen later on such that 
$B(x_{0}, 8h\kappa d) \subset \Omega.$ 
We will also use the cut-off function $ \zeta \in C_{c}^{\infty}(B_{2d}(x_{0}))$  such that 
\[  0\leq \zeta\leq 1 , \quad \zeta=1 \ \text{in} \ B_{d}(x_{0}), \quad |\nabla \zeta| \leq \frac{c}{d}, \quad |\nabla^{2} \zeta| \leq \frac{c}{d^{2}}, \]
where $\nabla ^{2}\zeta$ is the matrix of second derivatives.

We now introduce the function $u_* $ as $$u_{*} := u\zeta.$$ For $z\in B_{2\kappa d}(x_{0}) $ and any $0<R<2h\kappa d$, consider the homogeneous equation in $B_R(z)$:
\begin{equation}\label{eq1}
\left\{\begin{aligned}
\text{div} \langle \mathbb{A} \rangle_{R} \nabla v &= 0  \quad \text{in} \,\, B_R(z),\\
v-u_{*}&\in W^{1, r}_{0}(B_R(z)) 
\end{aligned}\right.
\end{equation} 
for some fixed  $r>1$. The next lemma estimates the difference $\nabla v - \nabla u_*$ as a function of  the mean oscillation of the coefficients.  
\begin{lemma}\label{comparison-estimate} 
For given $\gamma, r >1$ satisfying $1<r\gamma \leq p$  and any  $v \in W_0^{1,r}(B_R(z))$ solving \eqref{eq1}, there exist  positive constants $C$ and $\vartheta$ such that 
\[
\begin{split}
\fint_{B_R(z)}|\nabla v - \nabla u_{*}|^{r}\ dy \leq &\  C  \ \|\mathbb{A}\|^{1/\gamma'}_{*,2h\kappa d}\left(\fint_{B_R(z)}|\nabla u_{*}|^{r\gamma}\ dy \right)^{1/\gamma}\\
& + C \,(h\kappa)^{r}  \| \mathbb{A}\|_{*, 2h\kappa d}^{\vartheta} \fint_{B_R(z)}|\nabla u|^{r}\chi_{B_{2d}(x_{0})}(y)\ dy\\
&+ C \,(h\kappa)^{r}\fint_{B_R(z)}\left(|{\bf f}|^{r} + \left|\frac{u}{d}\right|^{r}\right)\chi_{B_{2d}(x_{0})}(y)\ dy. 
\end{split}
\] 
The constants $C=C(n,r, \gamma, \lambda, \Lambda) $ and  $\vartheta= \vartheta(r,n)$. 
\end{lemma}

\begin{proof}
Let $w := v-u_{*}$, then \eqref{eq1}  can be rewritten as 
\begin{equation*} 
\left\{\begin{aligned}
\text{div} \langle \mathbb{A} \rangle_{B_R(z)} \nabla w &= -\text{div} \langle \mathbb{A} \rangle_{B_R(z)} \nabla u_{*} \quad \text{in } B_R(z),\\
w &\in W^{1, r}_{0}(B_R(z)).
\end{aligned}\right.
\end{equation*}
Observing that 
\[
\text{div} \mathbb{A} \nabla (u \zeta) = \text{div} [{\bf f} \zeta + \mathbb{A} \nabla \zeta u] + (\mathbb{A}\nabla u - {\bf f})\cdot\nabla \zeta \quad \text{in } \mathcal{D}'(B_R(z)), 
\]
and using \eqref{interior-pde}, we obtain
 \begin{equation}\label{eq5}
 \left\{
\begin{aligned}
\text{div} \langle \mathbb{A} \rangle_{B_R(z)} \nabla w &= \text{div} (\mathbb{A} -\langle \mathbb{A} \rangle_{B_R(z)}) \nabla u_{*} - \text{div}[{\bf f}\zeta + \mathbb{A} \nabla \zeta u]  - \langle \mathbb{A} \rangle_{B_R(z)} \nabla u\cdot \nabla \zeta  \\
&\qquad - (\mathbb{A} - \langle \mathbb{A}\rangle_{B_R(z)} ) \nabla u \cdot \nabla \zeta  + {\bf f}\cdot \nabla \zeta \ \qquad \  \text{in} \ B_R(z) ,\\
w&=0 \quad \text{on} \ \partial B_R(z).
\end{aligned}\right.
 \end{equation}
 
Let $W^{-1,r'}(B_R(z))$ be the dual of $W^{1,r'}_{0}(B_R(z))$. Then
using $\text{spt}(\zeta) \subset B_{2d}(x_0)$ and standard elliptic estimates (see \cite[Theorem 2.1]{CW} and \cite[Corollary 1]{Fromm}), we find
\begin{equation}  \label{eqn_lemma_main}
\begin{split}
 \|\nabla w\|_{L^{r}(B_R(z))} &\leq C  \|(\mathbb{A} - \langle \mathbb{A} \rangle_{B_R(z)})\nabla u_{*}\|_{L^{r}(B_R(z))}  \\
& \quad + C \left[ \Big\|\left({\bf f} + \left|\frac{u}{d}\right|\right) \chi\Big\|_{L^{r}(B_R(z))} + J_{1} +J_{2} + J_{3}\right],
  \end{split}
	\end{equation}
 where we have set $\chi:= \chi_{B_{2d}(x_0)} $ and 
\begin{eqnarray*}
 J_{1} &:=&  \| (\mathbb{A} -  \langle \mathbb{A}\rangle_{B_R(z)})\nabla u\cdot \nabla \zeta\|_{W^{-1,r'}(B_R(z))},\\ 
 J_{2} &:=&  \|\langle \mathbb{A}\rangle_{B_R(z)}\nabla u\cdot \nabla \zeta\|_{W^{-1,r'}(B_R(z))}, \\
J_{3} &:=&  \|{\bf f} \cdot \nabla  \zeta\|_{W^{-1,r'}(B_R(z))} . 
\end{eqnarray*}

We  will now proceed with estimating the terms on the right hand side of \eqref{eqn_lemma_main}. 
To that end,  let $$\theta = \frac{n}{n-r'} > 1 \text{~~if~~} r' < n, \text{~~and~~} \theta =2 \text{~~if~~} r' \geq n.$$ Then by Sobolev embedding theorem, we have 
\begin{equation}\label{theta-sobolev}
\left(\fint_{B_R(z)} \phi^{\theta r'}\ dy \right)^{1/(\theta r')} \leq C R \left(\fint_{B_R(z)} |\nabla \phi|^{r'}\ dy \right)^{1/r'} \ \ \ \   \text{for all} \ \phi \in C_c^{\infty}(B_R(z)). 
\end{equation}
\subsection*{Estimate for $J_{1}$} By definition, we have  
\[
J_{1} = \sup\left\{\int_{B_R(z)} [(\mathbb{A}- \langle\mathbb{A}  \rangle_{B_R(z)})\nabla u\cdot \nabla \zeta] \phi\ dy:\  \phi \in C_c^{\infty} (B_R(z)),   \|\nabla \phi\|_{L^{r'}(B_R(z))} \leq 1 \right\}.
\] 
Now, applying H\"older's inequality with exponents, $\bar{\theta}=\frac{\theta r}{\theta r - \theta -r + 1}, r$ and  $\theta r'$, it follows from \eqref{theta-sobolev} and the fact that $\mathbb{A}\in L^\infty(\Omega)$  that 
\begin{equation*}\label{eq7}
\begin{split}
J_{1}&\leq \sup\left\{\frac{C}{d} \|\mathbb{A} - \langle \mathbb{A} \rangle_{B_R(z)}\|_{L^{\bar{\theta}}} \|\chi \nabla u\|_{L^{r}}\|\phi\|_{L^{\theta r'}}:\  \phi \in C_c^{\infty} (B_R(z)),   \|\nabla \phi\|_{L^{r'}} \leq 1   \right\}\\
& \leq \frac{C R}{d}  \|\mathbb{A}\|_{*,R}^{1/{\bar{\theta}}} \|\chi\nabla  u\|_{L^{r}},
\end{split}
\end{equation*}
where all norms are taken over the ball $B_R(z)$. 
\subsection*{Estimate for $J_{2}$} Let $\phi \in C_c^{\infty} (B_R(z))$ be such that $\|\nabla \phi\|_{L^{r'}(B_R(z))} \leq 1$, then by integrating by parts,  we have
\[
\begin{split}
\int_{B_R(z)}  \left(\langle \mathbb{A} \rangle_{B_R(z)} \nabla u \cdot \nabla \zeta \right) \phi \  dy & =  -\int_{B_R(z)}  u\  \text{div} ([\langle \mathbb{A} \rangle_{B_R(z)}\nabla \zeta] \phi )\ dy\\
&+   \int_{\partial B_R(z)} u(x)([\langle \mathbb{A} \rangle_{B_R(z)}\nabla \zeta] \phi  )\cdot \nu \ dH^{d-1}(x).
\end{split}
\]
The second term on the right hand side vanishes, since $\phi$ is compactly supported in $B_R(z)$.  For the  first term we find
\begin{equation}\label{eqII-1}
\begin{split}
 \left|\int_{B_R(z)}  u\  \text{div} (\langle \mathbb{A} \rangle_{B_R(z)}\nabla \zeta \phi )\ dy\right| &= \left|\int_{B_R(z)} u  \langle \mathbb{A} \rangle_{B_R(z)}^{T}: [\phi \nabla^{2}\zeta  + \nabla \zeta \otimes \nabla \phi ]\ dy\right|\\
 &\leq C \int_{B_R(z)} \left( |u| |\nabla^{2}\zeta| |\phi| \chi + |u||\nabla \zeta| |\nabla \phi|   \chi \right) \ dy.
\end{split}
\end{equation}
The first term on the right hand side of \eqref{eqII-1} can be estimated using Poincar\'e's inequality as 
\[
\begin{split}
\int_{B_R(z)}    |u| |\nabla^{2}\zeta| |\phi|  \chi \ dx & \leq  \frac{C}{d^{2}} \int_{B_R(z)} \chi |u||\phi| \ dy  \leq   \frac{C}{d^{2}}  \|\chi u\|_{L^{r}(B_R(z))}\|\phi\|_{L^{r'}(B_R(z))} \\
 &\leq \frac{C R}{d^{2}} \|\chi u\|_{L^{r}(B_R(z))}  \|\nabla \phi\|_{L^{r'}(B_R(z))} \leq \frac{C R}{d^{2}} \|\chi u\|_{L^{r}(B_R(z))}. 
\end{split}
\]
In a similar way, the second term on the right hand side of \eqref{eqII-1} can be estimated as follows:
\[
\begin{split}
\int_{B_R(z)}   |u||\nabla \zeta| |\nabla \phi| \chi\ dx  & \leq \frac{C}{d} \int_{B_R(z)} \chi |u||\nabla \phi|\ dy\\
& \leq  \frac{C}{d} \|\chi u\|_{L^{r}(B_R(z))}\|\nabla \phi\|_{L^{r'}(B_R(z))}  \leq  \frac{C}{d} \|\chi u\|_{L^{r}(B_R(z))}.
\end{split}
\]
Since we have $0<R<2h\kappa d$ and $h \kappa \geq 1$, we can now combine the last two estimates to obtain
\begin{equation}\label{eq8}
\begin{split}
J_{2}&=\sup\Big\{ \int_{B_R(z)}  \left(\langle \mathbb{A} \rangle_{B_R(z)} \nabla u \cdot \nabla \zeta \right) \phi \  dy:\  \phi \in C_c^{\infty} (B_R(z)),   \|\nabla \phi\|_{L^{r'}(B_R(z))} \leq 1 \Big\}\\
& \leq   C \left( \frac{R}{d}  + 1\right)\frac{1}{d} \|\chi u\|_{L^{r}(B_R(z))} \leq \frac{Ch\kappa }{d} \|\chi u\|_{L^{r}(B_R(z))}.
\end{split}
\end{equation}
\subsection*{Estimate for $J_{3}$} From the definition and Poincar\'e's inequality we obtain that 
\begin{equation}\label{eq9}
\begin{split}
J_{3} &= \sup\left\{ \left|\int_{B_R(z)} {\bf f} \cdot \nabla \zeta \phi \  dy  \right|:\  \phi \in C_c^{\infty} (B_R(z)),   \|\nabla \phi\|_{L^{r'}(B_R(z))} \leq 1 \right\}\\ 
& \leq \frac{C}{d}\sup\left\{  \|\chi {\bf f}\|_{L^{r}(B_R(z))} \|\phi\|_{L^{r'}(B_R(z)) }:\  \phi \in C_c^{\infty} (B_R(z)),   \|\nabla \phi\|_{L^{r'}(B_R(z))} \leq 1 \right\}\\
& \leq  \frac{C R}{d} \sup\left\{\|\chi {\bf f}\|_{L^{r}(B_R(z))}  \|\nabla \phi\|_{L^{r'}(B_R(z))}:\  \phi \in C_c^{\infty} (B_R(z)),   \|\nabla \phi\|_{L^{r'}(B_R(z))} \leq 1 \right\}\\
& \leq C h \kappa  \|\chi {\bf f}\|_{L^{r}(B_R(z))}. 
\end{split}
\end{equation}
\subsection*{Estimate for $\|(\mathbb{A} - \langle \mathbb{A} \rangle_{B_R(z)})\nabla u_{*}\|_{L^{r}(B_R(z))}$} Applying H\"{o}lder's inequality and the fact that $\mathbb{A}\in L^\infty(\Omega)$ we obtain that 
\begin{equation}
\label{eq10}
\begin{split}
\|(\mathbb{A} - \langle \mathbb{A} \rangle_{B_R(z)})\nabla u_{*}\|_{L^{r}(B_R(z))} &\leq \|\mathbb{A} - \langle \mathbb{A} \rangle _{B_R(z)}\|_{L^{r\gamma'}(B_R(z))}\|\nabla u_{*}\|_{L^{r\gamma}(B_R(z))} \\
&\leq C |B_R(z)|^{\frac{1}{r\gamma'}} \|\mathbb{A}\|_{*,R}^{\frac{1}{r\gamma'}} \|\nabla u_{*}\|_{L^{r\gamma}(B_R(z))}.
\end{split}
\end{equation}
Combining equations \eqref{eq7}, \eqref{eq8}, \eqref{eq9} and \eqref{eq10} into \eqref{eqn_lemma_main},  we get the desired estimate. This completes the proof of the Lemma. 
\end{proof}

 We will repeatedly use the following iteration device which can be found in \cite[Lemma 6.1]{Giu}.
\begin{lemma}
 \label{iteration}
Let $g(\tau) \geq 0$ be a bounded function in $[\tau_0, \tau_1]$ and $A, B \geq 0$ are given constants.  Suppose for any $\tau_0 \leq l_1 < l_2 \leq \tau_1$, we have 
$$g(l_1) \leq \ \theta\  g(l_2) + \frac{A}{(l_2-l_1)^{\alpha}} + B,$$ for some $\theta \in [0,1)$. Then for any $\tau_0 \leq l_1 < l_2 \leq \tau_1$, it holds that
$$g(l_1) \leq C(\alpha, \theta) \left[ \ \frac{A}{(l_2-l_1)^{\alpha}} + B\right]. $$
\end{lemma}

The following lemma gives a quantitative interior $C^{1,\alpha}$ estimate for equation \eqref{eq1}. 

\begin{lemma}\label{Integral-C1alpha}
Let
$v\in W_0^{1,r}(B_R(z))$ be  a solution to \eqref{eq1} for some $r\in (1,  p]$. 
Then there exists a constant $C=C(n,\lambda,\Lambda,r)>0$ and $\alpha = \alpha(n, \lambda,\Lambda) \in (0, 1)$ such that for every $\rho \in (0, R/2)$,  we have the estimate 
\begin{equation}\label{lemma4.5_estimate}
\fint_{B_{\rho}(z)}|\nabla v- \langle \nabla v\rangle_{B_{\rho}(z)} |\ dy \leq C\left( \frac{\rho}{R}\right)^{\alpha} 
\left( \fint_{B_{R}(z)}|\nabla u_{*}|^{r}\ dy \right)^{1/r}.
\end{equation}
\end{lemma}

\begin{proof}
Since $v\in W^{1,r}(B_R(z))$ and $r >1$,  by a result in \cite[Theorem A1.1]{BRE},  we see that $v\in W^{1, s}_{loc}(B_R(z))$ for all $s>1$. In particular, we have $v \in W^{1,2}_{loc}(B_R(z))$.  Thus by the standard  $C^{1, \alpha}$  estimate  (see e.g., \cite[Equation (3.6)]{KZ1} for $p=2$) we have 
\begin{equation}\label{standard-c1alpha}
\fint_{B_{\rho}(z)}|\nabla v - \langle \nabla v\rangle_{B_{\rho}(z)} |\ dy \leq C \left(\frac{\rho}{R}\right)^{\alpha} \left(\fint_{B_{R/4}(z)} |\nabla v|^{2}\ dy\right)^{1/2}
\end{equation}
for any $0 < \rho \leq R/4$ and the constant $C$ is independent of $\rho$.  Our first goal is to show that, via an intermediate exponent $s > 2$,
\begin{equation}
 \label{eq36}
 \left(\fint_{B_{R/4}(z)} |\nabla v|^{2}\ dy\right)^{1/2} \leq  \left(\fint_{B_{R/4}(z)} |\nabla v|^{s}\ dy\right)^{1/s} \leq C  \fint_{B_{R/2}(z)} |\nabla v|\ dy.
\end{equation}
To that end, by Gerhing's lemma (see \cite[Chapter 6]{Giu}) there exists  $s>2$ and a constant $C=C(n,\lambda,\Lambda,r)>0$   such that 
\begin{equation}\label{eq13}
\left(\fint_{B_{l}(\tilde{z})} |\nabla v|^{s}\ dy\right)^{1/s} \leq C \left(\fint _{B_{2l}(\tilde{z})}|\nabla v|^{2}\ dy\right)^{1/2}
\end{equation}
for any balls $B_{l}(\tilde{z}) \subset B_{2l}(\tilde{z}) \subset B_{R}(z)$. Let $R/4 < l_1 < l_2 < R/2$, then $B_{l_1}(z)\subset B_{l_2}(z)$. We shall now cover the ball $B_{l_1}(z)$ by a sequence of balls $B_{i} = B_{(l_2-l_1)/2}(\tilde{z}_{i})$ with $\tilde{z}_{i}\in B_{l_1}(z)$ in such a way that any point $y\in \mathbb{R}^{n}$ belongs to almost $N(n)$  balls of the collection $\{ 2 B_i \}:=\{B_{l_2-l_1}(\tilde{z}_i)\}  $, i.e., we have
\[
\sum_{i}\chi_{2B_{i} } (y)  \leq N = N(n) \qquad \forall  y \in \RR^n.   
\]
Note that $2B_{i}  = B_{l_2-l_1}(\tilde{z}_{i}) \subset B_{l_2}(z)\subset B_R(z)$ for any $i$.  
Therefore, applying \eqref{eq13} we get 
\[
\begin{split}
\int_{B_{i}}|\nabla v|^{s} \ dy  &\leq C \left(\int_{2B_{i}}|\nabla v|^{2}\ dy\right)^{s/2} |B_{i}|^{1- \frac{s}{2}}\\
&=C \left(\int_{2B_{i}} |\nabla v|^{2}dy\right)^{s/2} (l_2-l_1)^{\frac{(2-s)n}{2}}. 
\end{split}
\]
Summing over all $i$,  and using Minkowski's inequality since $s/2 > 1$, we obtain that 
\[
\begin{split}
\int_{B_{l_1}(z)}|\nabla v|^{s}\ dy &\leq C\sum_{i} \left(\int_{B_{l_2}(z)}\chi_{2B_{i}}(y)|\nabla v|^{2}\ dy \right)^{s/2}(l_2-l_1)^{\frac{(2-s)n}{2}}\\
& \leq C \left( \int_{B_{l_2}(z)} |\nabla v|^{2} \Big[\sum_{i}\chi_{2B_{i}}(y)\Big] \ dy \right)^{s/2} (l_2-l_1)^{\frac{(2-s)n}{2}}\\
& \leq C \left(N \int_{B_{l_2}(z)} |\nabla v|^{2}\ dy \right)^{s/2} (l_2-l_1)^{\frac{(2-s)n}{2}}.
\end{split}
\]
Interpolating between the space $L^{1}$ and $L^{s}$ then with $\varrho := \frac{s-2}{2(s-1)} \in (0, 1/2)$, we see that 
\[
\|\nabla v\|_{L^{2}(B_{l_2}(z))} \leq \|\nabla v \|_{L^{1}(B_{l_2}(z))}^{\varrho} \,\|\nabla v \|_{L^{s}(B_{l_2}(z))}^{1-\varrho}.
\]
 An application of Young's inequality now  gives
\[
\begin{split}
\int_{B_{l_1}(z)}|\nabla v|^{s} \ dy &\leq C \left(\int_{B_{l_2}(z)} |\nabla v|\ dy \right)^{\varrho s} \left(\int_{B_{l_2}(z)} |\nabla v|^{s}\ dy \right)^{1 - \varrho} (l_2-l_1)^{\frac{(2-s)n}{2}}\\
&\leq\frac{1}{2} \int_{B_{l_2}(z)}|\nabla v|^{s}\ dy  + C \left(\int_{B_{l_2}(z)} |\nabla v|\ dy \right)^{s} (l_2-l_1)^{n(1-s)}.
\end{split}
\]
The above inequality holds for any $R/4 \leq l_1 < l_2 \leq R/2$, for a constant $C$ that does not depend on $R, l$ or $z$ but only depends on $\lambda, \Lambda, n, r$.  
We can now apply  Lemma \ref{iteration} to obtain
\[
\int_{B_{R/4}(z)}|\nabla v|^{s}\  dy \leq C \left(\int_{B_{R/2}(z)}|\nabla v|\ dy\right)^{s} R^{n(1-s)}.
\]
We may rewrite the above inequality and use the fact that $s > 2$ to obtain \eqref{eq36} as desired.

Plugging  \eqref{eq36}  into \eqref{standard-c1alpha} and using H\"older's inequality, we obtain
\begin{equation} \label{eq14}
\fint_{B_{\rho}(z)}|\nabla v - \langle\nabla v\rangle_{B_{\rho}(z)} |\ dy \leq C \left(\frac{\rho}{R}\right)^{\alpha} \left(\fint_{B_{R/2}(z)}|\nabla v|^r \ dy\right)^{\frac{1}{r}}
\end{equation}
for all $0 < \rho < R/4$. It is obvious that \eqref{eq14} holds trivially when $R/4  < \rho\leq R/2$ as well. Thus \eqref{eq14} holds for all $\rho\in (0, R/2]$.  

Finally, we  apply standard $L^{r}$ estimates for linear equations with constant coefficients 
to obtain \eqref{lemma4.5_estimate}. This completes the proof of the lemma. 
\end{proof}

The following result combines Lemma \ref{comparison-estimate} and Lemma \ref{Integral-C1alpha} to yield a local mean oscillation estimate for the gradient of the solution $u $ to \eqref{interior-pde}. 
\begin{corollary}\label{mean-oscillation-interior}
Given  $1<r <p$,  there exist  positive constants $C$, $C_{0}$ and $\vartheta$ such that  for any $\rho \in (0, 2  \kappa d)$, $R= h\rho$, with $h\geq 2$,  and any $z\in B_{2\kappa d}(x_{0})$ we have 
\begin{equation} \label{mean-OSC-est} \begin{split}
\fint_{B_{\rho}(z)} &||\nabla u_{*}|- \langle|\nabla u_{*}|\rangle_{B_{\rho}(z)}|dy \leq C\|\mathbb{A}\|^{1/r- 1/p}_{*,2h\kappa d}\left(\fint_{B_{R}(z)}|\nabla u_{*}|^{p}dy \right)^{1/p}\\
& +  C\| \mathbb{A}\|_{*, 2h\kappa d}^{\frac{\vartheta}{r}} \left(\fint_{B_{R}(z)}|\chi_{B_{2d}(x_{0})}(y)\nabla u|^{r}dy\right)^{1/r}\\
&+ C\left(\fint_{B_{R}(z)} G(y)dy\right)^{1/r} +  C_{0} h^{-\alpha} 
\left( \fint_{B_{R}(z)}|\nabla u_{*}|^{r}\right)^{1/r}, 
\end{split}
\end{equation}
where $C=C(h)$, $C_{0}$ is independent of $h$, $\vartheta = \vartheta (r, n)$ and $\alpha\in (0, 1)$ is from Lemma \ref{Integral-C1alpha}. In \eqref{mean-OSC-est}, we set 
\begin{equation}\label{defn-G}
G(y) = \left( |{\bf f}(y)|^{r} +\left|\frac{u (y)}{d}\right|^{r}\right)\chi(y), \quad\quad \chi(y)=\chi_{B_{2d}(x_{0})}(y). 
\end{equation}

\end{corollary}
\begin{proof}
Let $\gamma =p/r >1$ and let $v$ be as in \eqref{eq1}. 
 By Lemma \ref{comparison-estimate}, we obtain a constant $C_{0}$ independent of $h$, $d$, and $x_0$ such that 
\begin{equation}\label{eq-comparison}
\begin{split}
\fint_{B_{R}(z)}|\nabla v - \nabla u_{*}|^{r}dx &\leq C_{0} \|\mathbb{A}\|^{1/\gamma'}_{*,2h\kappa d}\left(\fint_{B_{R}(z)}|\nabla u_{*}|^{r\gamma}dy \right)^{1/\gamma}\\
& + C_{0} (h\kappa )^{r}  \| \mathbb{A}\|_{*, 2h\kappa d}^{\vartheta} \fint_{B_{R}(z)}|\chi \nabla u|^{r}dy + C_{0} (h\kappa )^{r}\fint_{B_{R}(z)}G(y)dy, 
\end{split}
\end{equation}
where $\vartheta=\vartheta(r, n)>0$.  For any $0< \rho <2\kappa d$ and $z\in B_{2\kappa d}(x_{0})$ using triangle and H\"older's inequality, we have 
\begin{equation}\label{comparison-setup}
\begin{split}
\fint_{B_{\rho}(z)} &\left||\nabla u_{*}|- \langle|\nabla u_{*}|\rangle_{B_{\rho}(z)}\right|dy \\
&\leq 2\fint_{B_{\rho}(z)} \left||\nabla u_{*}|- \langle |\nabla v|\rangle_{B_{\rho}(z)}\right|dy \\
&\leq 2\fint_{B_{\rho}(z)} (|\nabla u_{*}-\nabla v|+ |\nabla v -\langle\nabla v\rangle_{B_{\rho}(z)}|)dy \\
&\leq 2\left(\fint_{B_{\rho}(z)} |\nabla u_{*}-\nabla v|^{r} dy\right)^{1/r} + 2\fint_{B_{\rho}(z)} |\nabla v -\langle\nabla v\rangle_{B_{\rho}(z)}|dy. 
\end{split}
\end{equation}
Now, since $R = h\rho$,  we apply Lemma \ref{Integral-C1alpha} to control the second term on the right hand side of \eqref{comparison-setup} by 
$$C_{0} h^{-\alpha} 
\left( \fint_{B_{R}(z)}|\nabla u_{*}|^{r}\right)^{1/r}, $$
 and then combining it with \eqref{eq-comparison}  to get the desired  estimate.  
\end{proof}

Now that we have established some estimating devices, we are ready to prove the main theorem of this section which gives a local interior estimate for the solution $u$ of \eqref{interior-pde}. 
\newline

\noindent \begin{proof}[Proof of Theorem \ref{main-interior-local-estimate}]  
Pick $r\in (1, p)$ as in  Corollary \ref{mean-oscillation-interior} to obtain the bound  \eqref{mean-OSC-est}
for any $\rho \in (0, 2\kappa  d)$, $R=h\rho\in (0,  2\kappa h d)$, with $h\geq 2$, and $z\in B_{2\kappa d}(x_{0})$, where $C  = C(h)$, $C_{0}$ is independent of $h$, and $\vartheta = \vartheta (r, n)$.

With   $G$ as in \eqref{defn-G},  we may now take  the supremum over $\rho \in (0, 2\kappa d)$ in \eqref{mean-OSC-est} to obtain the following pointwise estimate: 
\begin{equation}\label{sharp-max}
\begin{split}
\mathfrak{M}^{\#}_{2\kappa d}(|\nabla u_{*}|)(z) &\leq C(h) \bigg[\|\mathbb{A}\|^{1/r- 1/p}_{*,2h\kappa d} [\mathfrak{M}(|\nabla u_{*}|^{p})(z)]^{1/p}\\
&+ \| \mathbb{A}\|_{*, 2h\kappa d}^{\frac{\vartheta}{r}} [\mathfrak{M}(|\chi_{B_{2d}(x_{0})}\nabla u|^{r})(z)]^{1/r} \\
&+[\mathfrak{M}(G)(z)]^{1/r} \bigg]+ C_{0} \,h^{-\alpha}[\mathfrak{M}(|\nabla u_{*}|^{r})(z)]^{1/r},
\end{split}
\end{equation}
which holds for  all $z\in B_{2\kappa d}(x_{0})$.  We take the $q^{th}$  power on both sides of inequality \eqref{sharp-max} and then multiply by the weight function $w$ and integrate over the ball $B_{2d}(x_{0})$. 
Observing that $|\nabla u_{*}|$ is compactly supported in $B_{2d}(x_{0})$, we can apply Lemma \ref{lemma_sharp_phuc} to obtain
\begin{equation*}
\int_{B_{2d}(x_{0})} |\nabla u_{*}|^{q}w \leq C \int_{B_{2\kappa d}(x_{0})}[\mathfrak{M}^{\#}_{2\kappa d}(|\nabla u_{*}|)]^{q} w dx. 
\end{equation*}
It then follows from \eqref{sharp-max} that 
\[
\begin{split}
\int_{B_{2d}(x_{0})} |\nabla u_{*}|^{q}w dx
&\leq C(h) \bigg[\|\mathbb{A}\|^{q/r- q/p}_{*,2h\kappa d} \int_{\mathbb{R}^{n}}[\mathfrak{M}(|\nabla u_{*}|^{p})]^{q/p} w dx  \\
&+  \| \mathbb{A}\|_{*, 2h\kappa d}^{q \frac{\vartheta}{r}}  \int_{\mathbb{R}^{n}}[\mathfrak{M}(|\chi_{B_{2d}(x_{0})}\nabla u|^{r})]^{q/r} w dx \\
&+  \int_{\mathbb{R}^{n}}[\mathfrak{M}(G)]^{q/r} w dx \bigg]
+ C_{0} \,h^{-\alpha q}  \int_{\mathbb{R}^{n}}[\mathfrak{M}(|\nabla u_{*}|^{r})]^{q/r} w dx.
\end{split}
\]
 Noting that $q/p = q_{0}$, and so $w\in A_{q/p}\subset A_{q/r}$ for any $r\in (1, p)$, we may now use the boundedness of Hardy-Littlewood  maximal function on weighted $L^{p}$ spaces, Lemma \ref{weighted_maximal}, to get the inequality  
 \[
\begin{split}
\int_{B_{2d}(x_{0})} |\nabla u_{*}|^{q} w dx& \leq C (h) \bigg[\| \mathbb{A}\|_{*,2h\kappa d}^{q/r - q/p} \int_{B_{2d}(x_{0})}|\nabla u_{*}|^{q} w dx \\
 &+ \|\mathbb{A}\|_{*, 2h\kappa d}^{q \frac{\vartheta}{r}} \int_{B_{2d}(x_{0})}  |\nabla u|^{q} w dy  \\
&+  \int_{B_{2d}(x_{0})}(|{\bf f}|^{q}  + \left|\frac{u}{d}\right|^{q}) w dy \bigg] +C_{0} \ h^{-\alpha q} \int_{B_{2d}(x_{0})} |\nabla u_{*}|^{q} w dy.
 \end{split}
 \]
We now choose $h\geq 2$ large enough that 
\[
C_{0} h^{-\alpha q} \leq \frac{1}{2}. 
\]
This is possible since $C_{0}$ does not depend on $h$ and $\alpha \in (0, 1)$. 
We  can then absorb the third term on the right hand side of the above inequality to the left hand side. Once we do that $h$ will be fixed, and will depend  only on $\lambda, \Lambda, n, q, M_{0}$. 

Set $M = 8h\kappa$. Now for $d\leq \frac{\mathfrak{K}}{M}$, we have 
\[
2h\kappa d \leq \frac{\mathfrak{K}}{8h\kappa} 2h\kappa  = \mathfrak{K}/4. 
\]
As a consequence, whenever $ \|\mathbb{A}\|_{*, \mathfrak{K}/4}\leq \delta$, then 
 \[
\begin{split}
\int_{B_{2d}(x_{0})} |\nabla u_{*}|^{q} w dx& \leq C \left[ \delta^{q/r - q/p} \int_{B_{2d}(x_{0})}|\nabla u_{*}|^{q} w dx 
 +  \delta^{q\frac{\vartheta}{r}} \int_{B_{2d}(x_{0})} |\nabla u|^{q} w dy  \right. \\
&+ \left.  \int_{B_{2d}(x_{0})}\left(|{\bf f}|^{q}  + \left|\frac{u}{d}\right|^{q}\right) w dy\right].
 \end{split}
\]
Let $\delta_{1}>0$ be such that 
\[
C  \, \delta_{1}^{q/r - q/p} \leq 1/2. 
\]
Then for any $\delta_{0} \leq \delta_{1}$ such that $ \|\mathbb{A}\|_{*, \mathfrak{K}/4}\leq \delta_{0}$, we have
\begin{equation}\label{eq16}
\begin{split}
\int_{B_{d}(x_{0})} |\nabla u|^{q} w dx& \leq 
C\ \bigg[\delta_{0}^{q\frac{\vartheta}{r}} \int_{B_{2d}(x_{0})} |\nabla u|^{q} w dy  +  \int_{B_{2d}(x_{0})}(|{\bf f}|^{q}  + \left|\frac{u}{d}\right|^{q}) w dy\bigg],
 \end{split}
\end{equation}
where we used the fact that $\zeta = 1$ on $B_{d}(x_{0})$.
 
Let us recap  that \eqref{eq16} holds for all balls $B_{d}(x_{0})$ such that $d\leq \mathfrak{K}/M$ and $B_{Md}(x_{0})\subset \Omega$.  Next we work on to absorb the term involving $\int_{B_{2d}(x_{0})} |\nabla u|^{q} w dy$ in \eqref{eq16}.  
We will use an argument that was used previously  involving iteration and covering. To that end, let $d < l_{1} < l_{2} < 2d$. Cover $B_{l_{1}}(x_{0})$ by the collection 
$\{B_{i} =  B_{(l_{2}-l_{1})/2}(z_{i})\}$, where $z_{i}\in B_{l_{1}}(x_{0})$, in such a way that each point of $\mathbb{R}^{n}$ belongs to at most $N(n)$ balls of the collection $\{2B_{i}\}$. 
As $z_{i}\in B_{l_{1}}(x_{0})$ we have $2B_{i} = B_{l_{2}-l_{1}}(z_{i})\subset B_{l_{2}}(x_{0})$. Then using \eqref{eq16}, we get 
\[
\begin{split}
\int_{B_{i}} |\nabla u|^{q}w dx &\leq C  \delta_{0}^{q\frac{\vartheta}{r}} \int_{2B_{i}} |\nabla u|^{q}w dx+ C \int_{2B_{i}}\left(|{\bf f}|^{q}  + \Big|\frac{u}{l_{2}-l_{1}}\Big|^{q}\right) w dy.
 \end{split}
 \]
 Summing over $i$, we get 
 \[
 \int_{B_{l_{1}}(x_{0})}|\nabla u|^{q} w dx \leq N (n) C \left( \delta^{q\frac{\vartheta}{r}}_{0}  \int_{B_{l_{2}}(x_{0})} |\nabla u|^{q}w dx  +  \int_{B_{l_{2}}(x_{0})}\left(|{\bf f}|^{q}  + \Big|\frac{u}{l_{2}-l_{1}}\Big|^{q}\right) w dy\right). 
 \]
 Let now $\delta_{2} >0$ be such that 
 \[
 N (n) C \delta_{2}^{q\frac{\vartheta}{r}} \leq \frac{1}{2}
 \]
 and choose $\delta_{0} \leq \min\{\delta_{1}, \delta_{2}\}$. Then when $\|\mathbb{A}\|_{*,  \mathfrak{K}/4}\leq \delta_{0}$,  we have 
 \[
  \int_{B_{l_{1}}(x_{0})}|\nabla u|^{q} w dx \leq \frac{1}{2}  \int_{B_{l_{2}}(x_{0})} |\nabla u|^{q}w dx +  N (n) \ C \, \int_{B_{2d}(x_{0})}\left(|{\bf f}|^{q}  + \Big|\frac{u}{l_{2}-l_{1}}\Big|^{q}\right) w dy,
 \]
which holds for all $d < l_{1} < l_{2} < 2d$. Thus again applying the iteration lemma  (Lemma \ref{iteration}) we obtain that  
 \[
   \int_{B_{d}(x_{0})}|\nabla u|^{q} w dx \leq C \int_{B_{2d}(x_{0})}\left(|{\bf f}|^{q}  + \left|\frac{u}{d}\right|^{q}\right) w dy,
 \]
 and thus proving the theorem. 
\end{proof}
\section{Local boundary estimates} \label{local boundary}
In this section we prove a version of Theorem \ref{main-interior-local-estimate} over balls that intersect the boundary. We do this in two steps. The first step involves obtaining  the estimate for flat domains, and the second is for Lipschitz domains with small Lipschitz constant  using a flattening argument. 
\subsection{Local boundary estimates for equations over flat boundary}
We prove a boundary version of Theorem \ref{main-interior-local-estimate} to obtain a local estimate for solution of equations solved over half balls.  Let us introduce the notations $\mathbb{R}^{n}_{+} = \{(x', x_{n})\in \mathbb{R}^{n}: x_{n}>0\}$,   $\partial \mathbb{R}^{n}_{+} = \{(x', x_{n})\in \mathbb{R}^{n}: x_{n} = 0\}$ and  the half ball $B^{+}(x_{0}) = \mathbb{R}^{n}_{+} \cap B(x_{0})$, for $x_{0} \in \partial \mathbb{R}^{n}_{+}$. The main result of this subsection   is given in the following theorem. 
\begin{theorem}\label{flat-mainestimate}
Let  $x_{0} \in \partial \mathbb{R}^{n}_{+}$ be given. 
Suppose that $\mathfrak{K}>0$ , $M_{0} > 0$,  $1 <q<\infty$, and $w\in A_{q}$ such that $[w]_{A_{q}} \leq M_{0}$. Suppose also that 
$\mathbb{A}$ satisfies  \eqref{meas-symm} and \eqref{eq-Ellipticity},  ${\bf f}\in L^{q}_{w}(B_{\mathfrak{K}}^{+}(x_{0}),\mathbb{R}^{n})$ and  $u\in W^{1, q}_{w}(B_{\mathfrak{K}}^{+}(x_{0})$ is a weak solution of 
\begin{equation}
\label{mainhalfballequ}
\left\{
\begin{aligned}
\text{div}\, \mathbb{A}(x)\nabla u &= \text{div}\, {\bf f}(x)\quad \text{in\, $B_{\mathfrak{K}}^{+}(x_{0}) $}, \\
 u& =0\quad \text{on  $ B_{\mathfrak{K}}(x_0) \cap \partial \mathbb{R}^{n}_{+}$ }. 
 \end{aligned}
\right.
\end{equation}
Then there exist  $\delta_{0} > 0$, $M>2$, and $C>0$  such that 
if $\mathbb{A}$ is $(\delta, \mathfrak{K})$-BMO (over the set $\overline{B^{+}_{\mathfrak{K}}(x_0)}$)  with $\delta \leq \delta_{0}$, then   it holds that 
\[
\int_{B^{+}_{d}(x_{0}) } |\nabla u|^{q} w dx \leq  C \int_{B^{+}_{2d}(x_{0})} \left( |{\bf f}|^{q} + \left| \frac{u}{d}\right|^{q} \right)w dx
\]
for all $d \leq \mathfrak{K} /M$.  The constants $\delta_0$, $M$, and $C$ depend only on $\lambda, \Lambda,  n, q,$ and  $M_0$.
\end{theorem} 

\subsubsection{Estimates for homogeneous equations with constant coefficients near flat boundary}
Similar to the interior case, we  prove Theorem  \ref{flat-mainestimate} via comparison, comparing $u$  with a regular solution to a suitable homogeneous equation defined over half balls. For that  we will need various estimates for solutions $v$ of the following homogeneous equation: For some $R>0$, $\mathbb{A}_{0}$ an elliptic, symmetric constant matrix, $1<r<\infty$, and $u_{*} \in W^{1, r}(B_{R}^{+}(0))$ such that $u_{*}=0$ on $ B_{R}(0) \cap \partial \mathbb{R}^{n}_{+}$, let $v$ solve  
\begin{equation}\label{flat-eq1}
\left\{\begin{aligned}
\text{div} \, \mathbb{A}_{0}\nabla v &= 0\quad \text{in} \,\,B_{R}^{+}(0),\\
v - u_{*} &\in W^{1, r}_{0}(B_{R}^{+}(0)). 
\end{aligned}\right. 
\end{equation}
  Unlike the interior case however, up to the boundary estimate for the solutions to \eqref{flat-eq1} is not easy to find  in the literature especially when  $1 < r<2$.  Our goal is therefore to collect, and if necessary prove, estimates related to \eqref{flat-eq1} that will be useful for our comparison argument. Along this direction, the first result states that solutions to \eqref{flat-eq1} in fact belong to $W^{1, 2}$ well inside the half ball up to the flat boundary.      
  \begin{lemma}\label{W1r-solutions-energy-solution}
For a give $R>0$, $1 < r<\infty$, suppose that $u_{*} \in W^{1, r}(B_{R}^{+}(0))$ such that $u_{*}=0$ on $ B_{R}(0) \cap \partial \mathbb{R}^{n}_{+}$ and $v$ solves  \eqref{flat-eq1}. 
Then for any $0<\varepsilon < 1$,  $v\in W^{1, 2}(B_{\varepsilon R}^{+}(0))$ with the estimate 
\[
\|v\|_{W^{1, 2}(B_{\varepsilon R}^{+}(0))} \leq C_{\varepsilon} \| v\|_{W^{1, r}(B_{R}^{+}(0))}. 
\]
\end{lemma}
The interior version of Lemma \ref{W1r-solutions-energy-solution} is proved in \cite{BRE}, and it turns out that  
using the same duality argument as in \cite{BRE}, one can establish Lemma \ref{W1r-solutions-energy-solution}. We are able to actually prove a  more general version of Lemma 6.2 in which we allow the uniformly elliptic matrix $\mathbb{A}_{0}$ to have measurable coefficients with small BMO seminorm. We include a statement and proof in Appendix \ref{APP-A} for future reference.

We also need a {\it global}   estimate for the solution of \eqref{flat-eq1} that holds for all $1 < r < \infty$, where the constant does not depend on the radius of the half ball, $R$. 
This result is given in the following lemma and can be proved along the same line of proof of \cite[Corollary 1]{Fromm} (for Poisson equation), which utilized estimates for the corresponding Green's function. We note that the boundary of the half ball may have a large Lipschitz constant, but the convexity of the domain plays a crucial role to obtain the desired estimate.

\begin{lemma}
\label{Fromm-lemma}
Suppose that $\mathbb{A}_{0}$ is an elliptic, symmetric constant matrix. Then there exists a positive constant $C$ such that:  for any half ball $B^{+}_{R}(x_0)$ of radius $R$ and $1< r< \infty$, any 
${\bf f} \in L^{r}(B^{+}_{R}(x_0);\mathbb{R}^{n})$, there exists a unique solution $v\in W^{1,r}_{0}(B^{+}_{R}(x_0))$  to the equation 
\begin{equation}\label{fromm-eqn}
\text{div}\, \mathbb{A}_{0}\nabla v = \text{div}\,{\bf f} \quad \text{in } B_R^{+}(x_0)
\end{equation}
such that  
\begin{equation}\label{fromm-est}
\|\nabla v\|_{L^{r} (B_{R}^{+}(x_0))} \leq C\|{\bf f}\|_{L^{r}(B_{R}^{+}(x_0))}. 
\end{equation}
The constant $C$ is independent of  $R$ and $x_0$. 
\end{lemma}
\begin{proof} We first consider the case $x_0=0$ and $R=1$, i.e., the equation
\begin{equation}\label{wB1}
\left\{
\begin{aligned}
\text{div}\, \mathbb{A}_{0}\nabla w &= \text{div}\,{\bf g} \quad \text{in } B_{1}^{+}(0)\\
w&\in W^{1,r}_{0}(B_{1}^{+}(0)).
\end{aligned}\right. 
\end{equation}
Let $G(x,y)$ be the Green's function for the elliptic operator $- \text{div} \mathbb{A}_{0}\nabla\, \cdot$ in $B_1^{+}(0)$. Then the function
$$w(x)= \int_{B_1^{+}(0)} \nabla_{y} G(x,y) \cdot {\bf g}(y) dy $$
solves  the equation $\text{div}\, \mathbb{A}_{0}\nabla w = \text{div}\,{\bf g}$ in  $\mathcal{D}'(B_{1}^{+}(0)).$  Moreover, by the convexity of $B_1^+(0)$ one has the following weak-type bound
\begin{equation}\label{weak1-1}
 t |\{x\in B_1^+(0): |\nabla w(x)| > t\}| \leq C \|{\bf g}\|_{L^{1}(B_{1}^{+}(0))}\quad \forall t>0.
\end{equation}
Estimate \eqref{weak1-1} was obtained for the standard Laplacian, i.e., $\mathbb{A}_0=\mathbb{I}$, in \cite{Fromm} . One of the main ingredients in the proof of \eqref{weak1-1} in \cite{Fromm}  is pointwise estimates for the Green's function and its derivatives stated in \cite[Proposition 1]{Fromm}.  For general elliptic symmetric constant matrix
$\mathbb{A}_0$,  \cite[Proposition 1]{Fromm} still holds true by the work of Gr\"uter and Widman \cite{Gruter-Widman} and  thus the same  argument also yields \eqref{weak1-1}.
Using \eqref{weak1-1}, interpolation and duality 
we obtained a unique solution to equation  \eqref{wB1} along with the estimate 
\begin{equation}\label{strongr}
 \|w\|_{W_{0}^{1, r} (B_{1}^{+}(0))} \leq C \|{\bf g}\|_{L^{r}(B_{1}^{+}(0))}; 
\end{equation}
see \cite[Corollary 1]{Fromm}. 

 Next, noting that  the map $x\mapsto (x-x_0)/R$ is a one-to-one transformation of $B_R^{+}(x_0)$ onto $B_1^{+}(0)$, for a given ${\bf f}\in L^r(B^{+}_R(x_0))$, we define
${\bf g}(y)={\bf f}(R y+x_0)$ for $y\in B^{+}_1(0)$ and let $w$ be the unique solution of \eqref{wB1}. Then by the transformation $v(x)=R w((x-x_0)/R)$ and ${\bf f}(x)={\bf g}((x-x_0)/R)$ for $x\in B_{R}^{+}(x_0)$ we obtain a unique solution 
$v$ of \eqref{fromm-eqn} such that \eqref{fromm-est} holds with a constant $C$ independent of $R$ and $x_0$.
\end{proof}
\begin{remark}
 It is worth mentioning that the existence and uniqueness of a solution to equation  \eqref{wB1} along with the bound \eqref{strongr}  can also be obtained from \cite[Theorem 1.1]{Jia-Li-Wang}. 
\end{remark}

We now state the boundary analogue of Lemma \ref{Integral-C1alpha} that gives quantitative $C^{1,\alpha}$ regularity up to the boundary  for solutions of homogeneous equations. 
\begin{lemma} \label{flat-c1alpha}
Suppose that $R > 0$,  $x\in \mathbb{R}^{n}_{+}$ such that $B_{R}(x)\cap \partial \mathbb{R}^{n}_{+} \neq \emptyset$. Suppose also that $r > 1$,  $u_{*}\in W^{1, r}(B_{3R}(x) )$ and $u = 0$ on $B_{3R}(x)\cap \partial \mathbb{R}^{n}_{+} $.  
 Assume that $z\in \partial \mathbb{R}^{n}_{+}$ such that 
\[
B_{R}(x) \subset B_{2R}(z) \subset B_{3R}(x). 
\]
Then for any elliptic symmetric constant matrix $\mathbb{A}_{0}$ with constants of ellipticity 
$\lambda$ and $\Lambda$, there exist a constant $\alpha \in (0, 1)$ and a constant $C>0$ such that  if  $v$ solves 
\begin{equation}\label{flat-homo}
\left\{
\begin{aligned}
\text{div}\,\mathbb{A}_{0}\nabla v &= 0 \quad \text{in $B_{2R}^{+}(z)$},\\
v- u_{*}&\in W^{1, r}_{0}(B_{2R}^{+}(z)),
\end{aligned}\right. 
\end{equation}
then 
\[
\fint_{B_{\rho}(x)} |\nabla v - \langle \nabla v\rangle_{B_{\rho}(x)} | dy \leq C \left(\frac{\rho}{R}\right)^{\alpha} \left(\fint_{B_{3R}(x)} |\nabla u_{*}|^{r} dy\right)^{1/r}
\]
for any $0 <\rho < R/4$. The constants $\alpha=\alpha(n, \lambda, \Lambda)$ and $C=C(n, r, \lambda, \Lambda)$.
\end{lemma}
\begin{proof}
We begin by noting that $v$, a solution to \eqref{flat-homo},  is unique.  Moreover, by Lemma \ref{W1r-solutions-energy-solution}, $v\in W^{1, 2}(B_{\epsilon(2R)}^{+} (z))$  for any $\epsilon \in (0, 1)$. As a consequence if $\bar{v}$ solves 
\[
\left\{\begin{aligned}
\text{div} \, \mathbb{A}_{0}\nabla \bar{v} &= 0\quad \text{in $B_{R/4} (x)\cap \mathbb{R}^{n}_{+}$},\\
\bar{v}  - v &\in W ^{1, 2}_{0}(B_{R/4} (x)\cap \mathbb{R}^{n}_{+}),
\end{aligned}\right. 
\]
then $v = \bar{v}$  in $B_{R/4} (x)\cap \mathbb{R}^{n}_{+}$.  Thus, applying \cite[Lemma 3.7]{KZ2}  there exist  constants $C = C(n, \lambda, \Lambda)$ and $\alpha  = \alpha(n, \lambda, \Lambda)$ such that 
\begin{equation} \label{upto192}
\fint_{B_{\rho}(x)\cap \mathbb{R}^{n}_{+}} |\nabla v - \langle\nabla v \rangle _{B_{\rho}(x)}|dy \leq C \left( \frac{\rho}{R}\right)^{\alpha}\left( \fint_{B_{R/4}(x)\cap \mathbb{R}^{n}_{+}}|\nabla v|^{2} dy\right)^{1/2}
\end{equation}
 for any $0<\rho<R/192.$ For a possibly different constant $C$, inequality \eqref{upto192} is satisfied for all $ \rho \in (R/192, R/4]$ as well. Next we  show that 
 \begin{equation}\label{reverse-holder-flat}
 \left (\fint_{B_{R/4}(x)\cap \mathbb{R}^{n}_{+}}|\nabla v|^{2} dy\right)^{1/2} \leq C  \fint_{B_{R/2} (x)\cap \mathbb{R}^{n}_{+}} |\nabla v| dy.  
  \end{equation}
 Once we have \eqref{reverse-holder-flat},  then it follows from H\"older's inequality and Lemma \ref{Fromm-lemma} that 
\[
\begin{split}
  \left (\fint_{B_{R/2}(x)\cap \mathbb{R}^{n}_{+}}|\nabla v|^{2} dy\right)^{1/2}&\leq C  \fint_{B_{R} (x)\cap \mathbb{R}^{n}_{+}} |\nabla v| dy  \leq C \left(\fint_{B_{2R}^{+} (z)} |\nabla v|^{r} dy\right)^{1/r}\\
  & \leq C \left(\fint_{B^{+}_{2R} (z)} |\nabla u_{*}|^{r} dy\right)^{1/r} \leq C  \left(\fint_{B_{3R} (x)\cap 
	\mathbb{R}^{n}_{+}} |\nabla u_{*}|^{r} dy\right)^{1/r}, 
 \end{split}
 \]
 which completes the proof of the lemma. 
 To prove \eqref{reverse-holder-flat}, we proceed as in the proof of Lemma \ref{Integral-C1alpha} by interpolating between $L^{1}$ and $L^{p}$ for some $p>2$. To sketch the argument,  by higher integrability result, we have that there exists  $s>2$ and a constant $C>0$ such that 
 \[
 \left(\fint_{B_{\rho}(0)\cap \mathbb{R}^{n}_{+}}|\nabla v|^{s}dy \right)^{1/s} \leq C \left(\fint_{B_{2\rho}(0)\cap \mathbb{R}^{n}_{+}} |\nabla v|^{2}  dy \right)^{1/2}
 \] 
 for any ball $B_{\rho}(0)\subset B_{2\rho}(0) \subset B_{R}(x)$. Here the constant $C$ is independent of $0$ and $\rho$.  
 Rewriting the above inequality as 
  \[
 \left(\fint_{B_{\rho}(0)}|\nabla v|^{s} \chi_{\mathbb{R}^{n}_{+}}(y)dy \right)^{1/s} \leq C \left(\fint_{B_{2\rho}(0)} |\nabla v|^{2} \chi_{\mathbb{R}^{n}_{+}}(y)  dy \right)^{1/2}
 \] 
  for any ball $B_{\rho}(0)\subset B_{2\rho}(0) \subset B_{R}(x)$
we may  now use the iterative argument that was used in the proof of Lemma \ref{Integral-C1alpha} to obtain inequality \eqref{reverse-holder-flat}. 
 \end{proof}
\subsubsection{Mean-oscillation estimates over half balls}
\subsection*{Flat boundary set up} 
Suppose that $\mathfrak{K}>0$ , $M_{0} > 0$,  $1 <q<\infty$, and $w\in A_{q}$ such that $[w]_{A_{q}} \leq M_{0}$. Fix $x_{0} \in \partial \mathbb{R}^{n}_{+}$, and introduce the half ball $B_{\mathfrak{K} }^{+}(x_{0}) = \mathbb{R}^{n}_{+} \cap B_{\mathfrak{K} }(x_{0})$.  
Assume that ${\bf f} \in L^{q}_{w}(B_{\mathfrak{K} }^{+} (x_{0}))$ and $u \in  W^{1, q}_{w}(B^{+}_{\mathfrak{K} } (x_{0}))$ that solves equation \eqref{mainhalfballequ}.

With $\kappa > \sqrt{n}$ as in Lemma \ref{lemma_sharp_phuc},  let $B_{8 h \kappa  d}(x_{0}) \subset B_{\mathfrak{K} }(x_{0})$
 where $d>0$ and $h\geq 4$     will be determined later. 
We also use a cut-off function  $ \zeta \in C_{c}^{\infty}(B_{2d}(x_{0}))$  such that $0\leq \zeta\leq 1 $, $\zeta=1$ in $B_{d}(x_{0})$,  $|\nabla \zeta| \leq \frac{c}{d}$, and $|\nabla^{2} \zeta| \leq \frac{c}{d^{2}}$.  As before, set \[ u_{*} = u \zeta. \]
 Note that ${\bf f} \in L^{p}(B_{\mathfrak{K} }^{+}(x_{0}))$ and $u\in W^{1, p}(B^{+}_{\mathfrak{K} }(x_{0}))$ for $p = \frac{q}{q_{0}} > 1$ as in Corollary \ref{Lebesgue-vs-weighted}.  Moreover, since $u = 0$ on $ B_{\mathfrak{K} }(x_{0}) \cap \partial \mathbb{R}^{n}_{+}$ we can extend $u$ to be zero on $B_{\mathfrak{K} }^{-} (x_{0})$ to get that $u\in W^{1, p}(B_{\mathfrak{K} } (x_{0}))$. We also extend {\bf f} to be zero on $B_{\mathfrak{K} }^{-} (x_{0})$. Henceforth in this subsection we  work with these extended functions which will still be denoted by $u$ and {\bf f}.
 \begin{lemma}\label{flat-comparison-estimate-lemma}
Let $\gamma, r > 1$ be such that $1<\gamma r \leq p$.   Then there exist constants $C > 0$ and $\vartheta>0$ such that for $z\in \partial \mathbb{R}^{n}_{+}\cap B_{2\kappa d}(x_0)$, $0<R<2h\kappa d$,  $h\geq 4$, if  
 $v\in W^{1, r}(B^{+}_{2R}(z))$ is the solution of 
\begin{equation}
\label{vhalfball}
\left\{
\begin{aligned}
\text{div}\, \langle \mathbb{A}\rangle_{B^{+}_{2R}(z)} \nabla v &= 0 \quad \text{in $B^{+}_{2R}(z)$},\\
v- u_{*}&\in W^{1, r}_{0}(B^{+}_{2R}(z)), 
\end{aligned}\right. 
\end{equation}
then one has 
 \begin{equation}\label{flat-comparison-estimate}
\begin{split}
\left(\fint_{B^{+}_{2R}(z)} | \nabla v - \nabla u_{*}|^{r}dy \right)^{1/r}& \leq C  \|\mathbb{A}\|_{*_{\partial}, 4h\kappa d}^{1/r\gamma'}\left(\fint_{B^{+}_{2R}(z)} |\nabla u_{*}|^{r\gamma}\right)^{1/(r\gamma)}\\
&\quad +  C\, (h\kappa) \|\mathbb{A}\|_{*_{\partial}, 4h\kappa d} ^{\frac{\vartheta}{r}} \left( \fint_{B^{+}_{2R}(z)} |\nabla u|^{r} \chi_{B_{2d}(x_{0}) }\right)^{1/r} \\
&\quad + C\,(h\kappa) \left( \fint_{B^{+}_{2R}(z)} \left({\bf f}|^{r} + \left|\frac{u}{d}\right|^{r}\right)\chi_{B_{2d}(x_{0})}   \right)^{1/r}. 
\end{split}
\end{equation}
Here 
\[
 \|\mathbb{A}\|_{*_{\partial}, 4h\kappa d} = \sup_{y\in \partial \mathbb{R}^{n}_{+}\cap B_{2\kappa d}(x_0)} \  \sup_{0 < \rho < 4h\kappa d} \fint_{B^{+}_{\rho}(y)} |\mathbb{A}(x) - \langle \mathbb{A}  \rangle_{B^{+}_{\rho}(y)}|dx, 
\]
and the constants $C=C(n,r,\gamma, \lambda, \Lambda)$ and $\vartheta= \vartheta(r,n)$. 
 \end{lemma}
 \begin{proof}
 Clearly,  $\langle \mathbb{A}\rangle_{B^{+}_{2R}(z)}$ is a symmetric constant matrix that is elliptic with same ellipticity constants $\lambda$ and $\Lambda$.  
The difference  $w = v- u_{*}$ solves the equation 
\[
\left\{
\begin{aligned}
\text{div} \langle \mathbb{A}\rangle_{B^{+}_{2R}(z)} \nabla w &= -\text{div} \langle \mathbb{A}\rangle_{B^{+}_{2R}(z)} \nabla u_{*} \quad \text{in $B^{+}_{2R}(z)$},\\
w&\in W_{0}^{1, r}(B^{+}_{2R}(z)). 
\end{aligned}\right. 
\]
Similarly to how \eqref{eq5} was obtained, using \eqref{mainhalfballequ} we see that $w\in W_{0}^{1, r}(B^{+}_{2R}(z))$ solves the boundary value problem  
 \[
\left\{\begin{aligned}
\text{div} \langle \mathbb{A}\rangle_{B^{+}_{2R}(z)} \nabla w &= \text{div} (\mathbb{A}-\langle \mathbb{A}\rangle_{B^{+}_{2R}(z)}) \nabla u_{*} -\text{div}[{\bf f} \zeta + \mathbb{A} \nabla\zeta u] - \langle \mathbb{A}\rangle_{B^{+}_{2R}(z)}\nabla u\cdot \nabla \zeta\\
& \qquad -(\mathbb{A}-\langle \mathbb{A}\rangle_{B^{+}_{2R}(z)}) \nabla u\cdot \nabla \zeta 
+ {\bf f}\cdot \nabla \zeta \quad \text{in $B^{+}_{2R}(z)$}, \\
w &= 0\quad \text{on $B_{2R}^{+}(z)$}. 
\end{aligned}\right. 
\]
Now we apply Lemma \ref{Fromm-lemma} to get the estimate that  
  \begin{equation}\label{bdry-Fromm}
 \begin{split}
 \|\nabla w\|_{L^{r}(B^{+}_{2R}(z))} &\leq C \|(\mathbb{A} - \langle \mathbb{A} \rangle_{B^{+}_{2R}(z)})\nabla u_{*}\|_{L^{r}(B^{+}_{2R}(z))}  \\
& \quad +C \left[   \left\|\left(|{\bf f }| + \left|\frac{u}{d}\right|\right) \chi \right\|_{L^{r}(B^{+}_{2R}(z))}
 + J_{1} +J_{2} + J_{3}\right],
 \end{split}
  \end{equation} 
 where as before we  set $\chi := \chi_{B_{2d}(x_{0})}$. The terms $J_{1}, J_{2}$ and $J_{3}$ are given by 
 \[
 \begin{split}
 &J_{1}=  \| (\mathbb{A} -  \langle \mathbb{A} \rangle_{B^{+}_{2R}(z)})\nabla u\cdot \nabla \zeta\|_{W^{-1,r'}(B^{+}_{2R}(z))}, \\
& J_{2}:=  \|\langle \mathbb{A} \rangle_{B^{+}_{2R}(z)}\nabla u\cdot \nabla \zeta\|_{W^{-1,r'}(B^{+}_{2R}(z))},\\
&J_{3} :=  \|{\bf f} \cdot \nabla  \zeta\|_{W^{-1,r'}(B^{+}_{2R}(z))}. 
\end{split}
\]
We may now follow the exact procedure as in the proof of Lemma \ref{comparison-estimate} to estimate each term on the right hand side of \eqref{bdry-Fromm} and complete the proof of the lemma.
\end{proof}

The following corollary is an important consequence of the last two lemmas. 
\begin{corollary} \label{flatmeanOSC}
Let  $1< r < p$. Then there exist positive constants 
$C$, $C_{0}$, $\vartheta$  and $\alpha \in (0, 1)$ such that 
for any $x\in B_{2\kappa d}^{+}(x_{0}) $, $0<\rho<2\kappa d$, and $R=h\rho$ with $h\geq 4$,   we have  
\begin{equation}\label{flat-mean-oscillation}
\begin{split}
\fint_{B_{\rho}(x)} ||\nabla u_{*}|  - \langle |\nabla u_{*}|\rangle |dy& \leq C \|\mathbb{A}\|_{*, 4h\kappa d}^{1/r -1/p}\left(\fint_{B_{3R}(x)} |\nabla u_{*}|^{p}\right)^{1/p}\\
&\quad + C \|\mathbb{A}\|_{*, 4h\kappa d} ^{\frac{\vartheta}{r}} \left( \fint_{B_{3R}(x)} |\nabla u|^{r} \chi_{B_{2d}(x_{0}) }\right)^{1/r}\\
& \quad  + C \left( \fint_{B_{3R}(x) } G(y) \right)^{1/r}  + C_{0} h^{-\alpha} \left(\fint_{B_{3R}(x)} |\nabla u_{*}|^{r}\right)^{1/r}.
\end{split}
\end{equation}
 The constant $C =C(h)$ may depend on $h$ but $C_{0}$ is independent of $h$ and the constant $\vartheta = \vartheta (r, n)$. Here  $G$ is as defined in \eqref{defn-G} and  $\|\mathbb{A}\|_{*, 4h\kappa d}$ is defined over the set $\overline{B^{+}_{2\kappa d}(x_0)}$:
$$\|\mathbb{A}\|_{*, 4h\kappa d}= \sup_{y\in  \overline{B^{+}_{2\kappa d}(x_0)}} \ \sup_{0 < \rho < 4h\kappa d} 
\fint_{B_{\rho}(y)\cap \RR^n_{+}} |\mathbb{A}(x) - \langle \mathbb{A}  \rangle_{B_{\rho}(y)\cap \RR^n_{+}}|dx.
$$ 
\end{corollary}
\begin{proof}

We will consider the following two cases. 
\subsection*{Case 1}$B_{R}(x) \subset B^{+}_{\mathfrak{K}}(x_{0})$: In this case, we can proceed as in the interior case, Corollary \ref{mean-oscillation-interior}, to prove \eqref{flat-mean-oscillation}  even with $B_{3R}(x)$ replaced by $B_{R}(x)$ and $\|\mathbb{A}\|_{*, 4h\kappa d}$ replaced by $\|\mathbb{A}\|_{*, 2h\kappa d}$. 
\subsection*{Case 2} $B_{R}(x) \nsubseteq B^{+}_{\mathfrak{K}}(x_{0})$: Then $B_{R}(x)\cap \partial \mathbb{R}^{n}_{+} \neq \emptyset$. Let $z = z(x)$ be the point on  $ \partial \mathbb{R}^{n}_{+}\cap  B_{2\kappa d}(x_0)$  so that 
\[
|z - x| = \text{dist}(x, \partial \mathbb{R}_{+}^{n}).
\]
It is easy to see that  $|z -x| < R$ and therefore  
\begin{equation}\label{ztox}
B_{R}(x) \subset B_{2R}(z) \subset B_{3R}(x). 
\end{equation}

Let $\gamma=p/r>1$ and $v\in W^{1, r}(B^{+}_{2R}(z))$ be the unique solution of \eqref{vhalfball}. 
Applying Lemma \ref{flat-comparison-estimate-lemma}, we obtain constants $C$ and $\vartheta$ so that 
\eqref{flat-comparison-estimate} holds. 
Observe that since $B_{2R}(z) \subset B_{5 h \kappa  d}(x_{0}) \subset B_{\mathfrak{K} }(x_{0})$  and $u_{*}$ is zero on $B^{-}_{\mathfrak{K} }(x_{0})$ we may extend $v$ to be zero in  $B^{-}_{2R}(z)$ without affecting the inequality \eqref{flat-comparison-estimate}. We can then replace $B^{+}_{2R}(z)$ with $B_{2R}(z)$ in \eqref{flat-comparison-estimate}.  Moreover,
we can also replace $\|\mathbb{A}\|_{*_{\partial}, 4h\kappa d}$ with $\|\mathbb{A}\|_{*, 4h\kappa d}$ in \eqref{flat-comparison-estimate} as the latter is larger.

To obtain the estimate over $B_{R}(x)$, we use  the relation \eqref{ztox}  and write 
\begin{equation}\label{flat-comparison-estimate2}
\begin{split}
\left(\fint_{B_{R}(x)} | \nabla v - \nabla u_{*}|^{r}dy \right)^{1/r}& \leq C  \|\mathbb{A}\|_{*, 4h\kappa d}^{1/r\gamma'}\left(\fint_{B_{3R}(x)} |\nabla u_{*}|^{r\gamma}\right)^{1/(r\gamma)}\\
& \quad+ C\, (h\kappa) \|\mathbb{A}\|_{*, 4h\kappa d} ^{\frac{\vartheta}{r}} \left( \fint_{B_{3R}(x)} |\nabla u|^{r} \chi_{B_{2d}(x_{0}) }\right)^{1/r}\\
&\quad + C\,(h\kappa) \left( \fint_{B_{3R}(x)} G(y) \right)^{1/r} . 
\end{split}
\end{equation} 

On the other hand, with $x\in B^{+}_{2\kappa d} (x_{0})$ and $\rho\in (0, 2\kappa d)$, using the triangle and H\"older's inequality as in \eqref{comparison-setup}, we have  
\[
\begin{split}
\fint_{B_{\rho} (x)}& \left||\nabla u_{*}| - \langle |\nabla u_{*}|\rangle_{B_{\rho}(x)}\right|dy\\
&\leq 2 \left(\fint_{B_{\rho} (x)}|\nabla u_{*} - \nabla v|^{r}dy \right)^{1/r}+   2\fint_{B_{\rho} (x)} |\nabla v - \langle \nabla v\rangle_{B_{\rho}(x)}|dy.
\end{split}
\]

Finally, using the relation $R = h \rho$, $h\geq 4$, 
\eqref{flat-comparison-estimate2}, Lemma \ref{flat-c1alpha}, and  the above  estimate we get \eqref{flat-mean-oscillation} 
as desired.
\end{proof}

\begin{proof}[Proof of Theorem \ref{flat-mainestimate}]
The proof is similar to that of Theorem \ref{main-interior-local-estimate}. In fact, following the exact procedure and, using the mean oscillation estimate in Lemma \ref{flatmeanOSC}, we can show that there exists a constant $\delta_{1}>0$ such that  for any $\delta_{0} \leq \delta_{1}$ and $ \|\mathbb{A}\|_{*, \mathfrak{K}/2}\leq \delta_{0}$, we have
\begin{equation}\label{eq616}
\begin{split}
\int_{B_{d}(x_{0})} |\nabla u|^{q} w dx& \leq 
C\ \bigg[\delta_{0}^{q\frac{\vartheta}{r}} \int_{B_{2d}(x_{0})} |\nabla u|^{q} w dx  +  \int_{B_{2d}(x_{0})}\left(|{\bf f}|^{q}  + \left|\frac{u}{d}\right|^{q}\right) w dx\bigg],
 \end{split}
\end{equation}
 that holds for all  $d\leq \mathfrak{K}/M$ for some $M>2$. 
We should mention that from the interior estimate, by choosing $M$ large and  $\delta_{1}$ small  we have 
\begin{equation}\label{Int-in-boundary}
\begin{split}
\int_{B_{\rho}} |\nabla u|^{q} w dx& \leq 
C\ \int_{B_{2\rho}}\left(|{\bf f}|^{q}  + \left|\frac{u}{d}\right|^{q}\right) w dx,
 \end{split}
\end{equation}
that holds for all $B_{\rho}$ provided that $\rho\leq \mathfrak{K}/{M}$,  $B_{ M\rho} \subset B^{+}_{\mathfrak{K}}(x_0)$ and $ \|\mathbb{A}\|_{*, \mathfrak{K}/2}\leq \delta_{0}$. 
Next  using \eqref{eq616} and \eqref{Int-in-boundary}, we can absorb the first term on the right hand side of  \eqref{eq616} by a covering/iteration as before.  Indeed, let $d < l_{1} < l_{2} < 2d$, and  cover $B^{+}_{l_{1}}(x_{0})$ by the collection of balls that are either fully contained in $\mathbb{R}^{n}_{+}$ or whose center is the hyperplane $x_{n}=0$. To do so, we divide  $B^{+}_{l_{1}}(x_{0})$ in two regions. The first region is a layer of thickness $\frac{l_{2}-l_{1}}{4}$ near the hyperplane, and this region will be covered by balls centered at the hyperplane. We define this set explicitly as 
\[
\mathcal{L}(x_{0}) = B^{+}_{l_{1}}(x_{0})\cap\left\{(x', x_{n})\in \mathbb{R}^{n}_{+}: 0< x_{n}< \frac{l_{2}-l_{1}}{4}\right \}. 
\]
Now choose a collection of balls $\{B_{i} =  B_{(l_{2}-l_{1})/2}(z_{i})\}$, $z_{i}\in \partial \mathbb{R}^{n}_{+}\cap B_{l_{1}}(x_{0}) $ that cover $\mathcal{L}(x_{0})$ in such a way that each point of $\mathbb{R}^{n}$ belongs to at most $N = N(n)$ balls of the collection $\{2B_{i}\}$. 
As $z_{i}\in  \partial \mathbb{R}^{n}_{+}\cap B_{l_{1}}(x_{0})$ we have $2B_{i} = B_{l_{2}-l_{1}}(z_{i})\subset B_{l_{2}}(x_{0})$. Then using \eqref{eq616}, for each $i$ we have that 
\[
\begin{split}
\int_{B_{i}} |\nabla u|^{q}w dx &\leq C  \delta_{0}^{q\frac{\vartheta}{r}} \int_{2B_{i}} |\nabla u|^{q}w dx+ C \int_{2B_{i}}\left(|{\bf f}|^{q}  + \left|\frac{u}{l_{2}-l_{1}}\right|^{q}\right) w dy. 
 \end{split}
 \]
 Summing over $i$, we obtain 
 \[
 \int_{\mathcal{L}(x_{0})}|\nabla u|^{q} w dx \leq N (n) C \left[ \delta^{q\frac{\vartheta}{r}}_{0}  \int_{B_{l_{2}}(x_{0})} |\nabla u|^{q}w dx  +  \int_{B_{l_{2}}(x_{0})}\left(|{\bf f}|^{q}  + \left|\frac{u}{l_{2}-l_{1}}\right|^{q}\right) w dy\right]. 
 \]
 Let now $\delta_{2} >0$ be such that 
 \[
 N (n) C \delta_{2}^{q\frac{\vartheta}{r}} \leq \frac{1}{2}
 \]
 and choose $\delta_{0} \leq \min\{\delta_{1}, \delta_{2}\}$. Then when $\|\mathbb{A}\|_{* \frac{\mathfrak{K}}{2}}\leq \delta_{0}$,  we have 
 \begin{equation} \label{boundary-layer}
  \int_{\mathcal{L}(x_{0})}|\nabla u|^{q} w dx \leq \frac{1}{2}  \int_{B_{l_{2}}(x_{0})} |\nabla u|^{q}w dx +  N (n) \ C \, \int_{B_{2d}(x_{0})}\left(|{\bf f}|^{q}  + \left|\frac{u}{l_{2}-l_{1}}\right|^{q}\right) w dy,
 \end{equation}
that holds for all $d < l_{1} < l_{2} < 2d$.  Next,  we cover the remaining part $B^{+}_{l_{1}}(x_{0})\setminus \mathcal{L}(x_{0})$ by balls that are completely contained in $\mathbb{R}^{n}_{+}$. 
To do that, we choose balls $B_{i} :=B_{\frac{l_{2} - l_{1}}{10 M }}(z_{i})$, $z_{i} \in B^{+}_{l_{1}}(x_{0})\setminus \mathcal{L}(x_{0}) $ that cover $B^{+}_{l_{1}}(x_{0})\setminus \mathcal{L}(x_{0})$  in such a way that each point of $\mathbb{R}^{n}$ belongs to at most $N = N(n)$ balls of the collection $\{2B_{i}\}$. By construction, $M B_{i} = B_{\frac{l_{2} - l_{1}}{10 }}(z_{i})\subset B^{+}_{\mathfrak{K}}(x_0)$. We now apply inequality \eqref{Int-in-boundary} for each $B_{i}$ to obtain
\[
\int_{B_{i}} |\nabla u|^{q} w dx \leq  C \int_{2B_{i}}\left(|{\bf f}|^{q}  + \left|\frac{u}{l_{2}-l_{1}}\right|^{q}\right) w dx, 
\]
provided $ \|\mathbb{A}\|_{*, \mathfrak{K}/2}\leq \delta_{0}$.
Summing over $i$ we obtain that 
\begin{equation}\label{interior-layer}
\int_{B^{+}_{l_{1}}(x_{0})\setminus \mathcal{L}(x_{0})} |\nabla u|^{q} w dx \leq  C \,\int_{B_{2d}(x_{0})}\left(|{\bf f}|^{q}  + \left|\frac{u}{d}\right|^{q}\right) w dx
\end{equation}
We now add inequalities \eqref{boundary-layer} and \eqref{interior-layer} to obtain the following estimate (recall that $\nabla u=0$ in $B^{-}_{l_{1}}(x_{0})$):
 \[
  \int_{B_{l_{1}}(x_{0})}|\nabla u|^{q} w dx \leq \frac{1}{2}  \int_{B_{l_{2}}(x_{0})} |\nabla u|^{q}w dx +  C  \, \int_{B_{2d}(x_{0})}\left(|{\bf f}|^{q}  + \left|\frac{u}{l_{2}-l_{1}}\right|^{q}\right) w dy
 \]
 that holds for all $d < l_{1} < l_{2} \leq2d$ provided $ \|\mathbb{A}\|_{*, \mathfrak{K}/2}\leq \delta_{0}$.
Thus again applying   Lemma \ref{iteration} we obtain that  
 \[
   \int_{B_{d}(x_{0})}|\nabla u|^{q} w dx \leq C \int_{B_{2d}(x_{0})}\left(|{\bf f}|^{q}  + \left|\frac{u}{d}\right|^{q}\right) w dy,
 \]
 which proves the theorem. 
\end{proof}
\subsection{Local  up to the boundary estimates for Lipschitz domains}

\begin{theorem}\label{Lip-boundary-main}
Suppose that $\Omega\subset \mathbb{R}^{n}$ is a bounded  domain with Lipschitz boundary. Let $M_{0} > 0$, $\mathfrak{K} >0$, $1 <q<\infty$ and $w\in A_{q}$ such that $[w]_{A_{q}} \leq M_{0}$. Fix $x_{0}\in \partial\Omega$. Then there exist   $\delta_{0} >0$, $M >8$, and $C>0$ such that  for ${\bf f}\in L^{q}_{w}(\Omega\cap B_{\mathfrak{K}}(x_{0}) )$ and  any weak solution $u\in W^{1, q}_{w}(\Omega\cap B_{\mathfrak{K}}(x_{0}))$ to the problem 
\begin{equation}\label{basic-PDE-bdry}
\left\{
\begin{aligned}
\text{div}\, \mathbb{A}(x)\nabla u &= \text{div}\, {\bf f}(x)\quad \text{in\, $\Omega$}, \\
 u& =0\quad \text{on  $\partial \Omega\cap B_{\mathfrak{K}}(x_{0})$}, 
 \end{aligned}
\right.
\end{equation}
one has the estimate 
\begin{equation*}\label{estimate-boundary}
\int_{B_{d/2}(x_{0})\cap \Omega} |\nabla u|^{q} w dx \leq C \int_{B_{4 d}(x_{0})\cap \Omega}\left(|{\bf f}|^{q} + \left|\frac{u}{d}\right|^{q}\right)w dx 
\end{equation*}
for all $d \in (0, \mathfrak{K}/M]$,
provided $\mathbb{A}$ is $(\delta, \mathfrak{K})$-BMO and $\Omega$ is  $(\delta, \mathfrak{K})$-Lip with $\delta \leq \delta_{0}$. 
The constants $\delta_0, M,$ and $C$ depend only on $n, q, \lambda, \Lambda,$ and $M_{0}$. 
\end{theorem}
\begin{proof}
We are going to use standard flattening of the boundary procedure to  prove the theorem. First, we flatten the boundary and transform the equation to be set on a half ball. Along the way, we will discover that the small Lipschitz constant of the boundary will allow the transformed coefficients to have small BMO seminorm.  We then apply estimates on half balls that are developed in the previous subsection. 

\subsection*{{\bf Flattening the boundary} }
First, since $\Omega$ is a $(\delta, \mathfrak{K})$-Lip domain,  for each $x_{0}\in \partial \Omega,$ there correspond a coordinate system with  $x = (x', x_{n})$ where $x' \in\mathbb{R}^{n-1}$ and $x_{n}\in \mathbb{R}$ and a Lipschitz continuous function $\Gamma:\mathbb{R}^{n-1} \to \mathbb{R}$ such that 
\[
\begin{split}
\Omega \cap B_{\mathfrak{K}}(x_{0}) &= \{(x', x_{n}): \Gamma(x') < x_{n}\}\cap B_{\mathfrak{K}}(x_{0}),\\
\partial \Omega\cap B_{\mathfrak{K}}(x_{0}) &= \{(x', x_{n}): \Gamma(x') = x_{n}\}\cap B_{\mathfrak{K}}(x_{0}).
\end{split}
\]
  Moreover,  $\|\nabla_{x'}\Gamma\|_{L^{\infty}} <\delta \leq 1$. 
Define  the flattening mapping $\Phi: \mathbb{R}^{n} \to \mathbb{R}^{n}$ as 
\[
y = \Phi(x) = \Phi(x' ,x_{n}) := (x', x_{n}-\Gamma(x')),  
\]
and  its inverse $\Psi: \mathbb{R}^{n} \to \mathbb{R}^{n}$ as 
\[x =\Psi(y)= \Phi^{-1}(y) := (y', y_{n} + \Gamma(y')). 
\]
It is then clear that the gradient matrices $\nabla \Phi$ and $\nabla \Psi$ are inverses of each other. Moreover, after defining the the vector $\vec{\mathfrak{l}}(x) := (\nabla_{x'}\Gamma(x'),0) $,  and $\vec{{\bf e}}_{n} := (0, 0, \dots, 1)$, a simple calculation shows that 
\[\nabla \Phi(x) = \mathbb{I} - \vec{\bf e}_{n}\otimes \vec{\mathfrak{l}}(x) \quad \text{for a.e. } x\in \mathbb{R}^{n}, 
\]
 and
\[ \nabla \Psi(y) = 
 \mathbb{I} +\vec{{\bf e}}_{n}\otimes \vec{\mathfrak{l}}(\Phi^{-1}(y)) \quad \text{for  a.e. }  y\in \mathbb{R}^{n}. 
 \] 
In the above,  $\mathbb{I}$ represents the identity matrix and $\otimes$ is the dyadic product.  In particular, 
 \begin{equation}\label{det}
 \det(\nabla \Psi) = \det \nabla \Phi  =1.
 \end{equation}
 
Next we observe that  
\begin{equation}\label{Bmathfrak}
B_{\mathfrak{K}/2}(\Phi(x_0))\subset \Phi(B_{\mathfrak{K}}(x_0))\subset B_{2\mathfrak{K}}(\Phi(x_0)).
\end{equation}
Indeed, for $y\in B_{\mathfrak{K}/2}(\Phi(x_0))$ we have $|y-\Phi(x_0)|<\mathfrak{K}/2$ and $y=\Phi(x)$ with $x=\Psi(y)$.
Thus 

\[
\begin{split}
|x-x_0|=|\Psi(y)- \Psi(\Phi(x_0))|& \leq |y-\Phi(x_0)|+ \| \nabla_{x'} \Gamma\|_{L^\infty}|y'-[\Phi(x_0)]'|\\
\leq 2|y-\Phi(x_0)| <\mathfrak{K}. 
\end{split}
\]
That is, $y\in B_{\mathfrak{K}}(\Phi(x_0))$, which yields the first inclusion. Arguing similarly, we obtain the second inclusion.

By  \eqref{Bmathfrak} we have   
\begin{equation}\label{inc}
B^{+}_{\mathfrak{K}/2}(\Phi(x_{0})) \subset \Phi(\Omega\cap B_{\mathfrak{K}}(x_{0}))\,\text{and}\, B_{\mathfrak{K}/2}^{-}(\Phi(x_{0}))\subset \Phi(\Omega^{c}\cap B_{\mathfrak{K}}(x_{0})).
\end{equation}
  Now define 
\[
u_{1} (y) = u(\Psi(y)) \quad \text{for } y\in B^{+}_{\mathfrak{K}/2}(\Phi(x_{0})). 
\]
\subsubsection*{\bf Observation 1} The function $u_{1}$ is a weak solution to the equation 
\begin{equation*} 
\left\{
\begin{aligned}
\text{div}\, \mathbb{A}_{1}(y) \nabla u_{1}(y) &= \text{div}\, {\bf f}_{1}(y)\quad \text{in $B^{+}_{\mathfrak{K}/2}(\Phi(x_{0}))$},\\
u_{1} &= 0\quad \text{on} \,\,B_{\mathfrak{K}/2}(\Phi(x_{0})) \cap \partial \mathbb{R}^{n}_{+},
\end{aligned}
\right.
\end{equation*}
where $$\mathbb{A}_{1}(y) = \nabla \Phi(\Psi(y))\mathbb{A}(\Psi(y)) [\nabla \Phi(\Psi(y)]^{T} \quad\text{and}\quad {\bf f}_{1}(y) = [\nabla \Phi(\Psi(y)]^{T} {\bf f}(\Psi(y)). $$  
Indeed, for any smooth function $\varphi\in C_{c}^{\infty}(B^{+}_{\mathfrak{K}/2}(\Phi(x_{0})))$, we have that, after change of variables $x = \Psi(y) \iff y = \Phi(x)$,
\[
\begin{split}
&\int_{B^{+}_{\mathfrak{K}/2}(\Phi(x_{0}))} \langle \mathbb{A}_{1}(y)\nabla u_{1}(y), \nabla \varphi(y) \rangle dy  \\
&= \int_{\Psi(B^{+}_{\mathfrak{K}/2}(\Phi(x_{0}))) }\langle \nabla \Phi(x)\mathbb{A}(x) [\nabla \Phi(x)]^{T}[\nabla \Psi( \Phi(x))]^{T} \nabla u (x),  [\nabla \Phi(x)]^{-T}\nabla (\varphi(\Phi(x))) \rangle dx\\
& = \int_{\Psi(B^{+}_{\mathfrak{K}/2}(\Phi(x_{0}))) }\langle \mathbb{A}(x) \nabla u (x),  \nabla (\varphi(\Phi(x))) \rangle dx \\
&= \int_{\Psi(B^{+}_{\mathfrak{K}/2}(\Phi(x_{0}))) }\langle {\bf f}(x),  \nabla (\varphi(\Phi(x))) \rangle dx = \int_{B^{+}_{\mathfrak{K}/2}(\Phi(x_{0}))} \langle {\bf f}_{1}(y), \nabla \varphi(y) \rangle dy.
\end{split}
\]
Here we  used \eqref{det}, \eqref{inc}, the fact that $ u$ is a weak solution of \eqref{basic-PDE-bdry}, and that the function
$\varphi(\Phi(\cdot))  \in C_{0}^{0, 1}(\Psi(B^{+}_{\mathfrak{K}/2}(\Phi(x_{0})) )$ is a valid test function for \eqref{basic-PDE-bdry}. 
\subsubsection*{\bf Observation 2} As in \eqref{Bmathfrak} we have $\Psi(B_r(y))\subset B_{2r}(\Psi(y))$ for all balls $B_r(y)\subset\RR^n$. Thus since  $\Psi$ and $\Phi$ are measure preserving maps, it can easily be shown that $w_{1}(y) = w(\Psi(y))$ is also an $A_{q}$ weight with $[w_1]_{A_q}\leq c\, [w]_{A_q} $.  Similarly,  $u_{1} \in W^{1, q}_{w_{1}}(B_{s}^{+}(\Phi(x_{0}))) $, ${\bf f}_{1}\in L^{q}_{w_{1}}(B^{+}_{s}(\Phi(x_{0})))$, and the coefficient matrix $\mathbb{A}_{1}$ is uniformly elliptic. To verify the later, let $y\in B^{+}_{s}(\Phi(x_{0})$, $\xi \in \mathbb{R}^{n}$, and $\eta = [\nabla \Phi(\Psi(y))]^{T}\xi$. Then it follows from the ellipticity of $\mathbb{A}$ that  
\[
\lambda |\eta|^{2} \leq \langle \mathbb{A}_{1}(y)\xi, \xi \rangle = \langle \mathbb{A}(\Psi (y))[\nabla \Phi(\Psi(y))]^{T}\xi, [\nabla \Phi(\Psi(y))]^{T}\xi \rangle  \leq \Lambda |\eta|^{2}. 
\]
To estimate $|\eta|$ in terms of $|\xi|$, we observe that 
\[
|\eta|^{2} = |\xi - \xi_{n}\mathfrak{l}(x)|^{2} \leq 2 |\xi|^{2} + 2|\xi_{n}|^{2}\|\nabla_{x'}\Gamma\|^{2}_{L^{\infty}} \leq 2(1 + \|\nabla_{x'}\Gamma\|^{2}_{L^{\infty}})|\xi|^{2} \leq 4|\xi|^{2}. 
\] 
Similar calculations yield that 
$
|\xi|^{2} \leq 2(1 + \|\nabla_{x'}\Gamma\|^{2}_{L^{\infty}}) |\eta|^{2} \leq 4|\eta|^{2}, 
$
from which we conclude that $\mathbb{A}_{1}$ is uniformly elliptic on $B^{+}_{s}(\Phi(x_{0})$  with 
 constants of ellipticity $\lambda/4$ and $4\Lambda$.

\subsubsection*{\bf Observation 3} We can control the BMO seminorm of  $\mathbb{A}_{1}$ in terms of the BMO seminorm of $\mathbb{A}$  and the Lipschitz constant of $\partial \Omega$. 
Writing  $\mathbb{A}_{1}$ in  the expanded form we have that for any $y\in B_{\mathfrak{K}/2}^{+}(\Phi(x_0))$,
\[
\begin{split}
\mathbb{A}_{1}(y) &= \mathbb{A}(\Psi (y)) - [\mathbb{A}(\Psi (y))\vec{\mathfrak{l}}(\Psi (y))\otimes \vec{\bf e}_{n}] - [\vec{\bf e}_{n}\otimes \mathbb{A}(\Psi (y))\vec{\mathfrak{l}}(\Psi (y))]\\
& \quad + \langle  \mathbb{A}(\Psi (y))\vec{\mathfrak{l}}(\Psi (y)), \vec{\mathfrak{l}}(\Psi (y))\rangle (\vec{\bf e}_{n}\otimes \vec{\bf e}_{n}). 
\end{split}
\]
It then follows from direct calculations that 
\[
\begin{split}
\| \mathbb{A}_{1}\|_{*, \frac{\mathfrak{K}}{4}} &:=\sup \fint_{B_{\rho}(y_{0})\cap \mathbb{R}^{n}_{+}}  |\mathbb{A}_{1}(y) - \langle\mathbb{A}_{1}\rangle _{B_{\rho}(y_{0}) \cap \mathbb{R}^{n}_{+}}| dy \\
 &\leq \sup
\fint_{B_{\rho}(y_{0})\cap \mathbb{R}^{n}_{+}}  \left|\mathbb{A}(\Psi(y)) - \langle\mathbb{A}(\Psi (\cdot))\rangle _{B_{\rho}(y_{0}) \cap \mathbb{R}^{n}_{+}} \right| dy  + C(\Lambda) \|\nabla_{x'} \Gamma\|_{L^{\infty}},
\end{split}
\]
where the suprema are taken over all $y_0\in \overline{B^{+}_{\mathfrak{K}/4}(\Phi(x_{0}))}$ and $\rho\in (0, \mathfrak{K}/4)$.

By adding and subtracting any constant matrix $\boldsymbol{\mu}_{\rho}$, which will be properly determined shortly,  and by making a change of variables we obtain 
\begin{equation}
\label{psibrhoy}
\begin{split}
\|\mathbb{A}_{1}\|_{*, \frac{\mathfrak{K}}{4}} & \leq 2 \sup
\fint_{\Psi(B_{\rho}(y_{0})\cap \RR^n_{+})\cap \Omega}  |\mathbb{A}(x) - \boldsymbol{\mu}_{\rho}| dx  + C(\Lambda) \|\nabla_{x'} \Gamma\|_{L^{\infty}},
\end{split}
\end{equation}
where again the supremum is taken over all $y_0\in \overline{B^{+}_{\mathfrak{K}/4}(\Phi(x_{0}))}$ and $\rho\in (0, \mathfrak{K}/4)$.

Now as in \eqref{Bmathfrak} we have $\Psi(B_\rho(y))\subset B_{2\rho}(\Psi(y))$, and thus 
$$\Psi(B_{\rho}(y_{0})\cap \RR^n_{+})\cap \Omega \subset B_{2\rho}(\Psi(y_{0}))\cap \Omega.$$
 Also,  similar calculations show that 
\[
y_{0}\in \overline{B^{+}_{\mathfrak{K}/4}(\Phi(x_{0}))} \Longrightarrow \Psi(y_{0}) \in \overline{B_{\mathfrak{K}/2}(x_{0})\cap \Omega}. 
\]
Therefore,  after plugging $\boldsymbol{\mu}_{\rho} = \langle\mathbb{A}\rangle_{B_{2\rho}(\Psi(y_{0}))}$ into \eqref{psibrhoy}
and setting $z_0=\Psi(y_{0})$  we have that 
\[
\begin{split}
\|\mathbb{A}_{1}\|_{*, \frac{\mathfrak{K}}{4}} & \leq 2^{n+1} \sup 
\fint_{B_{2\rho}(z_0))\cap \Omega}  |\mathbb{A}(x) -  \langle\mathbb{A}\rangle_{B_{2\rho}(z_0)\cap \Omega }| dx   + C(\Lambda) \|\nabla_{x'} \Gamma\|_{L^{\infty}},
\end{split}
\]
where now the supremum is taken over all $z_0\in \overline{B_{\mathfrak{K}/2}(x_{0}) \cap \Omega}$ and $\rho\in(0,\mathfrak{K}/4)$.
This yields  
\begin{equation}\label{A1bmovsA}
\|\mathbb{A}_{1}\|_{*, \frac{\mathfrak{K}}{4}} \leq 2^{n+1}\| A\|_{*, \mathfrak{K}} + C(\Lambda) 
\|\nabla _{x'}\Gamma\|_{L^{\infty}}. 
\end{equation}
\subsection*{{\bf Local estimates at the Lipschitz  boundary} }  
We now apply Theorem \ref{flat-mainestimate} to the  problem 
\begin{equation*} 
\left\{
\begin{aligned}
\text{div}\, \mathbb{A}_{1}(y) \nabla u_{1}(y) &= \text{div}\, {\bf f}_{1}(y)\quad \text{in $B^{+}_{\mathfrak{K}/4}(\Phi(x_{0}))$},\\
u_{1} &= 0,\quad \text{on} \,\,B_{\mathfrak{K}/4}(\Phi(x_{0})) \cap \partial \mathbb{R}^{n}_{+}
\end{aligned}
\right.
\end{equation*}
to conclude that there exists a constant $\delta_{0}>0$ such that whenever $\mathbb{A}$ is $(\delta, \mathfrak{K})$-BMO and $\Omega$ is   $(\delta, \mathfrak{K})$-Lip with $\delta \leq \delta_{0}$, there exist a constant $M>2$ and $C>0$  such that 
\[
\int_{B^{+}_{d}(\Phi(x_{0})) } |\nabla u_{1}|^{q} w_{1} dy \leq  C \int_{B^{+}_{2d}(\Phi(x_{0}))} \left( |{\bf f}_{1}|^{q} + \left| \frac{u_{1}}{d}\right|^{q} \right)w_{1} dy
\]
for all $d \leq \mathfrak{K}/(4M)$.  Note that the smallness of BMO seminorm of $\mathbb{A}$  and that of the Lipschitz constant of $\partial \Omega$ imply the smallness of the BMO seminorm of $\mathbb{A}_{1}$ which follows from 
the bound \eqref{A1bmovsA} in {\bf Observation 3}.

Finally, after making the change of variables $x = \Psi(y)$, and noting that $\Psi(B^{+}_{d}(\Phi(x_{0}))\supset B_{d/2}(x_{0})\cap \Omega$, and $\Psi(B^{+}_{2d}(\Phi(x_{0})))\subset B_{4d}(x_{0})\cap \Omega$  we finally obtain 
\[
\int_{B_{d/2}(x_{0})\cap\Omega } |\nabla u|^{q} w dx \leq  C \int_{B_{4d}(x_{0})\cap \Omega } \left( |{\bf f}|^{q} + \left| \frac{u}{d}\right|^{q} \right)w dx
\]
for all $d \leq \mathfrak{K}/4M$. 
 \end{proof}
 \section{Global gradient estimates for Lipschitz domains }\label{global}
 In this section we prove the main results of the paper. 
\subsection{Proof of Theorem \ref{maintheorem-Cacciaopoli}} Let $M_1>8$ be the largest $M$ that are obtained in Theorem \ref{main-interior-local-estimate} and Theorem \ref{Lip-boundary-main}.
Let $\Omega_{\mathfrak{K}/(20M_1)}=\{x\in\Omega: {\rm dist}(x, \partial\Omega)> \mathfrak{K}/(20M_1)\}$ and use Vitali 
covering lemma to cover it by $N_1$ balls   $B_i$ of radius  $\mathfrak{K}/(60M_1^{2})$ so that the collection $\{(1/3)B_i\}$ are disjoint. Note that we have $M_1B_i \subset\Omega$, and   $N_1=N_1(n, {\rm diam}(\Omega)/\mathfrak{K})$ by a volume comparison. Thus by applying Theorem \ref{main-interior-local-estimate} to each $B_i$ with $d=\mathfrak{K}/(60M_1^{2})$ and summing over $i$ we get
$$\int_{\Omega_{\mathfrak{K}/(20M_1)}} |\nabla u|^q w dx \leq N_1 C \int_{\Omega} (|{\bf f}|^q + |u/\mathfrak{K}|^q) w dx.$$

Similarly, we can also cover $\Omega\setminus \Omega_{\mathfrak{K}/(20M_1)}$ by $N_2=N_2(n, {\rm diam}(\Omega)/\mathfrak{K})$
balls  $B_j=B_{\mathfrak{K}/(2M_1)}(\xi_j)$ with centers $\xi_j\in \partial\Omega$. Then 
applying Theorem \ref{Lip-boundary-main} to each boundary ball $B_{\mathfrak{K}}(\xi_{i})\cap \Omega$ with $d=\mathfrak{K}/M_1$ and summing over $j$ we get 
$$\int_{\Omega\setminus \Omega_{\mathfrak{K}/(20M_1)}} |\nabla u|^q w dx \leq N_2 C \int_{\Omega} (|{\bf f}|^q + |u/\mathfrak{K}|^q) w dx.$$

Finally, combining the last two estimates we obtain inequality \eqref{stability-Cacciaopoli} 
with a constant $C=C(\lambda, \Lambda, q, n, M_{0},  \text{diam}(\Omega)/\mathfrak{K})$ as desired.
 
\subsection {Proof of Corollary \ref{weightedLorentz}}
For any  $w\in A_{q}$ such that $[w]_{A_{q}} \leq M_{0}$, using the open ended property of $A_{q}$ weights, Lemma \ref{open_ended},  there exists $\epsilon > 0$, $q-\epsilon >1$ such that $w\in A_{q-\epsilon}$ and $[w]_{A_{q-\epsilon}} \leq C([w]_{q})$. Also $w\in A_{q +\epsilon}$ and $[w]_{A_{q + \epsilon}} \leq [w]_{A_{q}}$
Now applying Theorem \ref{maintheorem}, there exist positive constants $C_{1}, C_{2}$ and $\delta$ such that  whenever $\mathbb{A}$ is $(\delta, \mathfrak{K})$-BMO and $\Omega$ is a $(\delta, \mathfrak{K})$-Lip domain, the solution-gradient operator $\mathcal{T} : {\bf f}\mapsto \nabla u$ is a well defined {\it linear operator} on both $L^{q-\epsilon}_{w}(\Omega;\mathbb{R}^{n})$ and $L^{q+\epsilon}_{w}(\Omega;\mathbb{R}^{n})$. Moreover, $\mathcal{T}$  satisfies the estimates  
\[
\|\mathcal{T}({\bf f})\|_{L^{q-\epsilon}_{w} (\Omega)} \leq C_{1} \|{\bf f}\|_{L^{q-\epsilon}_{w}(\Omega)}\quad \forall {\bf f}\in L^{q-\epsilon}_{w}(\Omega;\mathbb{R}^{n})
\]
and 
\[
\|\mathcal{T}({\bf f})\|_{L^{q+\epsilon}_{w} (\Omega)} \leq C_{2} \|{\bf f}\|_{L^{q+\epsilon}_{w}(\Omega)}\quad \forall {\bf f}\in L^{q+\epsilon}_{w}(\Omega;\mathbb{R}^{n}). 
\]
Given $\theta_{0}\in (0, 1)$ we can now apply the interpolation theorem \cite[Theorem 1.4.19]{Gra} with $\frac{1}{p} := \frac{1-\theta_{0}}{q-\epsilon} + \frac{\theta_{0}}{q + \epsilon}$ to obtain the estimate 
\[
\|\mathcal{T}({\bf f})\|_{L^{p, r}_{w} (\Omega)} \leq C \|{\bf f}\|_{L_{w}^{p, r}(\Omega)}\quad \forall {\bf f}\in L^{p, r}_{w}(\Omega;\mathbb{R}^{n})  \, \text{and $0 < r\leq \infty$}. 
\] 
 In particular, if we choose $\theta_{0} = \frac{q + \epsilon}{2q} \in (0, 1)$ we obtain that $p=q$ and 
\[
\|\nabla u\|_{L^{q,r}_{w}(\Omega)}  = \|\mathcal{T}({\bf f})\|_{L^{q, r}_{w} (\Omega)} \leq C \|{\bf f}\|_{L_{w}^{q, r}(\Omega)} \quad \forall {\bf f}\in L^{q, r}_{w}(\Omega;\mathbb{R}^{n})  \, \text{and $0 < r\leq \infty$}, 
\]
provided $\mathbb{A}$ is $(\delta, \mathfrak{K})$-BMO and $\Omega$ is a $(\delta, \mathfrak{K})$-Lip domain as claimed.

 \subsection{Proof of Theorem \ref{Lorentz-Morrey} }
 We will show that the theorem actually follows from Corollary \ref{weightedLorentz}, after choosing an appropriate choice of weight functions as in the proof of \cite[Theorem 2.3]{M-P2}. We sketch its proof here, referring \cite{M-P2} for details. Suppose that ${\bf f}\in \mathcal{L}^{q, r; \theta}(\Omega;\mathbb{R}^{n})$ where $1 < q < \infty$ and $0 < t < \infty$.
 For any $z\in \Omega$, $0 < \rho \leq \text{diam}(\Omega) $ and for any $\varepsilon\in (0, \theta)$, consider the weight 
 \[
 w_{z} (x) = \min\{|x - z|^{-n + \theta - \varepsilon}, \rho^{-n +\theta-\varepsilon}\}.
 \]
 Then for each $z$, $w_{z}\in A_{q}$ for any $q\in (1, \infty)$  (see \cite[Chapter 9]{Gra}) and 
 \[
 [w_{z}]_{A_{q}} \leq C(n, q,  \theta). 
 \] 
 for some constant $C(n, q, r, \theta)$ independent of $z$ and $\rho$. 
 On the one hand, since $w_{z}(x) \equiv \rho^{-n + \theta-\varepsilon}$ on $B_{\rho}(z)$ we have that 
 \[
 \| \nabla u\|_{L^{q, r}(B_{\rho}(z)\cap \Omega)}^{r}  = \rho^{\frac{(n-\theta+ \varepsilon)r}{q}}  \| \nabla u\|_{L^{q, r}_{w_{z}}(B_{\rho}(z)\cap \Omega)}^{r} \leq C \rho^{\frac{(n-\theta+ \varepsilon)r}{q}}  \|{\bf f}\|_{L^{q, r}_{w_{z}}( \Omega)}^{r}  
 \]
 where we have used Corollary \ref{weightedLorentz} (with $M_{0} = C(n, q,  \theta)$) provided the coefficient matrix has small BMO seminorm and the Lipschitz constant of $\Omega$ is also small.     
 On the other hand, it turns out that for a constant $C$ that depends only on $q, r, $ and $n$
 \[
 \|{\bf f}\|_{L^{q, r}_{w_{z}}( \Omega)}^{r}  \leq C \|{\bf f}\|^{r}_{\mathcal{L}^{q, t; \theta}} \rho ^{\frac{-\varepsilon \ r}{q}};
 \]
 see the proof of \cite[Theorem 2.3]{M-P2} for details. Combining the above inequalities we have 
 \[
 \| \nabla u\|_{L^{q, r}(B_{\rho}(z)\cap \Omega)}^{r} \leq C \|{\bf f}\|^{r}_{\mathcal{L}^{q, t; \theta}} \rho ^{\frac{(n- \theta)r}{q}}, 
 \]
which is valid for all $z \in \Omega$ and $\rho\in (0, \text{diam} (\Omega)]$ from which the desired estimate follows.  

\appendix

\section{Very weak solutions in $W^{1, p}$, $p > 1$, are finite energy solutions} \label{APP-A}
In this appendix we would like to demonstrate the validity of Brezis's result \cite[Theorem A1.1]{BRE} up to the boundary.   Brezis's result \cite[Lemma A.1]{BRE} says that very weak solutions of homogeneous linear equation with continuous coefficients that are in $W^{1, p}$ for some $p > 1$ are in fact  in $W^{1, q}_{loc}$ for any $q>1$.  In this appendix we will show that in fact the statement will remain true for boundary value problems even with coefficients with small BMO and posed over  half balls, having a zero boundary condition on the flat part of the boundary.     The proof strictly follows the argument used in the proof of  \cite[Lemma A.1]{BRE}, with natural modification to  fit our setting. 
The main tool we use is the following lemma which is actually the main result of \cite{Jia-Li-Wang07}. The result is stated in its general form, to include what are called `quasiconvex domains', see \cite[Definition 3.2]{Jia-Li-Wang07}. For our purpose we simply note that polygonal convex domains, sector of balls, and ball segments (such as half balls) are all quasiconvex domains. 
\begin{lemma}\cite[Theorem 1.1]{Jia-Li-Wang07}\label{a-Jia-Li-Wang07}
Let  $1 < p < \infty$  and ${\bf f} \in L^{p}(\Omega;\mathbb{R}^{n})$.
Suppose that  $\mathbb{A}$ is a  symmetric, uniformly elliptic matrix with constants of ellipticity $\lambda$ and $\Lambda$.  
Then there exists $\delta=\delta(n, p, \lambda, \Lambda ) > 0$ such that whenever  $\mathbb{A}$ is $(\delta, \mathfrak{K})$-BMO and $\Omega$ is a  $(\delta, \sigma, \mathfrak{K})$-quasiconvex bounded domain, the Dirichlet problem  
\[
\left\{\begin{aligned}
\text{div}\,\mathbb{A}(x)\nabla u(x) &= \text{div}\,{\bf f} \,\, \text{in }  \Omega, \\
u&= 0\quad \text{on } \partial \Omega,
\end{aligned} \right. 
\]
has a unique solution $u\in W^{1, p}_{0}(\Omega)$. Moreover, there exists $C=C(\lambda, \Lambda, n, p, \sigma, \mathfrak{K}, \Omega) > 0$ such that 
\[
\|\nabla u\|_{L^{p}(\Omega)} \leq C \|{\bf f}\|_{L^{p}}. 
\] 
\end{lemma}
The main theorem we would like to prove is the following.  We note that we have used the result with constant coefficients in the form given in Lemma \ref{W1r-solutions-energy-solution}. 
\begin{theorem}
Let $1 < r < \infty$, $R > 0$, and $\mathfrak{K} > 0$. Suppose that  $\mathbb{A}$ is a  symmetric, uniformly elliptic matrix with constants of ellipticity $\lambda$ and $\Lambda$. Suppose also that $u_{*} \in W^{1, r}(B_{R}^{+}(0))$ such that $u_{*}=0$ on $ B_{R}(0) \cap \partial \mathbb{R}^{n}_{+}$.
 Then there exists  $\delta=\delta(n, r, \lambda, \Lambda) > 0$ such that if $\mathbb{A}$ is  $(\delta, \mathfrak{K})$-BMO and $v$ solves  
\begin{equation}\label{A-flat-eq1}
\left\{\begin{aligned}
\text{div} \, {\mathbb{A}}(x)\nabla v &= 0\quad \text{in} \,\,B_{R}^{+}(0),\\
v - u_{*} &\in W^{1, r}_{0}(B_{R}^{+}(0)),
\end{aligned}\right.
\end{equation}
then for any $0<\tau < 1$, one has $v\in W^{1, 2}(B_{\tau R}^{+}(0))$ along with the estimate 
\[
\|v\|_{W^{1, 2}(B_{\tau R}^{+}(0))} \leq C_{\tau} \| v\|_{W^{1, r}(B_{R}^{+}(0))}. 
\]
The constant $C_{\tau}$  depends only  $\lambda, \Lambda, \tau,$ $r$, $R$, $n$ and $\mathfrak{K}$. 
\end{theorem}
\begin{proof}
There is nothing to prove if $r \geq 2$. So we assume that $1 < r <2$.  
Suppose that ${\bf g} \in C_{c}^{\infty}(B_{R}^{+}(0))$ such that 
\[
\|{\bf g}\|_{L^{s'}(B_{R}^{+}(0))} \leq 1, 
\]
where $\frac{n}{n-1}<s\leq  2$  to be determined and $1/s  + 1/s' = 1$. 

We now apply Lemma \ref{a-Jia-Li-Wang07} to obtain $\delta  > 0$ such that a unique solution $w\in W^{1, 2}_{0}(B_{R}^{+}(0)) \cap W_{0}^{1, s'}(B_{R}^{+}(0))$  to  $\text{div}\, {\mathbb{A}} \nabla w = \text{div}\, {\bf g}$ in $B_{R}^{+}(0)$ 
exists such that
\begin{equation}\label{2-s'-estimates} 
 \quad \|w\|_{W^{1, s'}(B_{R}^{+}(0))} \leq C \|{\bf g}\|_{L^{s'}(B_{R}^{+}(0))} \leq C,
\end{equation}
provided $\mathbb{A}$ is $(\delta, \mathfrak{K})$-BMO. 
Note also that by choosing $\delta$ even smaller, we can have  $w\in W^{1, r'}(B_{R}^{+}(0)) $ where $r'$ is the H\"older conjugate exponent of $r$.
By definition of $w$, we have that 
\[
\int_{B_{R}^{+}(0)} {\mathbb{A}}\nabla w \cdot \nabla \phi dx = \int_{B_{R}^{+}(0))} {\bf g}\cdot \nabla \phi dx, \quad \quad \forall \phi\in C_{c}^{1}(B_{R}^{+}(0)). 
\]
Moreover,  by density the above equation is valid for all $\phi\in W^{1, r}_{0}(B_{R}^{+}(0))$.   
Now for any fixed $\zeta\in C_{c}^{\infty}(B_{R}(0))$, we can take $\phi(x) = \zeta(x) v(x) $ as a test function in the above, since  by assumption $v\in W^{1, r}(B_{R}^{+}(0))$  and $v = 0$ on $B_{R}(0) \cap \partial \mathbb{R}^{n}_{+}$, and therefore the product $\phi(x) = \zeta(x) v(x) \in W^{1, r}_{0}(B_{R}^{+}(0))$.   We then have the following:
\begin{equation}\label{a-w-equation}
\int_{B_{R}^{+}(0))} {\mathbb{A}} \nabla w\cdot (\zeta \nabla v + v \nabla \zeta) dx = \int_{B_{R}^{+}(0))} {\bf g}\cdot (\zeta \nabla v + v \nabla \zeta )dx.
\end{equation}
Again, since $v\in W^{1, r}(B_{R}^{+}(0)) $ solves equation \eqref{A-flat-eq1},  it follows from the definition of $v$ as a solution and density   
\[
\int_{B_{R}^{+}(0)}{\mathbb{A}}\nabla v \cdot \nabla \psi dx =0,\quad \quad \forall \psi \in W_{0}^{1,r'}(B_{R}^{+}(0)).
\]
Now take $\psi = \zeta w \in W_{0}^{1, r'}(B_{R}^{+}(0))$ as a test function in the above to get that 
\begin{equation}\label{a-v-equation}
\int_{B_{R}^{+}(0)}{\mathbb{A}}\nabla v \cdot (w\nabla \zeta  + \zeta \nabla w)   dx =0. 
\end{equation}
Comparing equations \eqref{a-w-equation} and \eqref{a-v-equation} we obtain that 
\[
\begin{split}
\int_{B_{R}^{+}(0)} \zeta \nabla v \cdot {\bf g} dx &=  -\int_{B_{R}^{+}(0)} w\left({\mathbb{A}} \nabla v \cdot \nabla \zeta \right)dx + \int_{B_{R}^{+}(0)}v \left( {\mathbb{A}}\nabla w \cdot \nabla \zeta \right)dx - \int_{B_{R}^{+}(0)} v {\bf g}\cdot \nabla \zeta dx\\
&=I_{1} + I_{2} + I_{3}
\end{split}
\]
Since $r<2\leq n$, and by the Sobolev embedding, 
\[
\|v\|_{L^{r^{*}} (B_{R}^{+}(0))} \leq C \|v\|_{W^{1, r}(B_{R}^{+}(0))},\quad\text{where } \quad \frac{1}{r^{*}} = \frac{1}{r} - \frac{1}{n}.
\]
We estimate  $I_{1}$,  $I_{2}$ and $I_{3}$, by choosing $s\in (n/(n-1), 2]$ in the following way:  if $r^{*} \leq 2$, then choose $s = r^{*}$, if $r^{*}> 2$, choose $s=2$.  
\subsection*{\bf Case 1} $r^{*} > 2$: In this case choose $s=2$, ($s' = 2$ ). Then we have $w\in W^{1, 2}(B_{R}^{+}(0)) $ and from  \eqref{2-s'-estimates} we have 
\[
\|w\|_{W^{1, 2}(B_{R}^{+}(0))} \leq C. 
\]
Now when $n\geq 3$,   the assumption $r^{*} > 2$ is equivalent to  $r' < 2^{*}$; when $n = 2$, Sobolev imbedding  implies that 
$W^{1, 2}(B_{R}^{+}(0)) \hookrightarrow L^{q}(B_{R}^{+}(0))$ for any  $q\in [2, \infty)$. Combining the two we find that 
\[
\|w\|_{L^{r'}(B_{R}^{+}(0))} \leq C.
\]
With this at hand, we can now estimate $|I_{i}|$, $i= 1, 2, 3$. To that end,
\[
|I_{1}| \leq \int_{B_{R}^{+}(0)} |w||{\mathbb{A}} \nabla v||\nabla \zeta| dx \leq C \|\nabla \zeta\|_{L^{\infty}} \| w\|_{L^{r'}(B_{R}^{+}(0)) } \| \nabla v\|_{L^{r}(B_{R}^{+}(0))} \leq C \|\nabla v\|_{L^{r}(B_{R}^{+}(0))}. 
\]
We also notice from Sobolev embedding that 
\[
\|v\|_{L^{2}(B_{R}^{+}(0))} \leq C \|v\|_{L^{r^{*} } (B_{R}^{+}(0))} \leq C \| v\|_{W^{1, r}(B_{R}^{+}(0))}. 
\]
As a consequence, 
\[
|I_{2}| \leq \int_{B_{R}^{+}(0))} |v||{\mathbb{A}}\nabla w| |\nabla \zeta|dx \leq C \|\nabla \zeta \|_{L^{\infty}} \|v\|_{L^{2}(B_{R}^{+}(0))} \|\nabla w\|_{L^{2}(B_{R}^{+}(0))} \leq C \|v\|_{W^{1, r}(B_{R}^{+}(0))}, 
\]
and 
\[
|I_{3}| \leq C \|\nabla \zeta\|_{L^{\infty}}\|{\bf g}\|_{L^{2}}\| v\|_{L^{2}}  \leq C  \|v\|_{W^{1, r}(B_{R}^{+}(0))}. 
\]
Combining the estimates for $|I_{i}|$, we get that 
\[
\left|\int_{B_{R}^{+} (0)} \zeta \nabla v \cdot {\bf g} dx\right| \leq C \|v\|_{W^{1, r}(B_{R}^{+}(0))}, \quad \text{provided $\|{\bf g}\|_{L^{2}} \leq 1$}.  
\]
In particular, by choosing the cut off function $\zeta$ appropriately,  for a given $\tau > 0$, $v\in W^{1,2}(B_{\tau R}^{+}(0))$ and 
\[
\|v\|_{W^{1,2} (B_{\tau R}^{+}(0))} \leq C_{\tau} \|v\|_{W^{1, r}(B_{R}^{+}(0))}. 
\]

\subsection*{Case 2}$r^{*} \leq 2$: In this case, take $s = r^{*}$. We then have that
\[
\frac{1}{r'} = 1- \frac{1}{r} = \frac{1}{r^{*}} + \frac{1}{s'} - \frac{1}{r} = \frac{1}{s'} - \frac{1}{n}. 
\]
That is,  the Sobolev conjugate of $s'$ is $r'$ and that $s' = \frac{r^{*}}{r^{*} - 1} < n$. 
As a consequence, by \eqref{2-s'-estimates} and Sobolev embedding again we have 
\[
\|w\|_{L^{r'} (B_{R}^{+}(0))} \leq C \|w\|_{W^{1, s'}(B_{R}^{+}(0))} \leq C. 
\]
We again use this to estimate $|I_{i}|,$ $i=1, 2, 3$. We begin with $I_{1}$: 
\[
|I_{1}| \leq \int_{B_{R}^{+}(0)} |w||{\mathbb{A}} \nabla v||\nabla \zeta| dx \leq C \|\nabla \zeta\|_{L^{\infty}} \| w\|_{L^{r'}(B_{R}^{+}(0)) } \| \nabla v\|_{L^{r}(B_{R}^{+}(0))} \leq C \|\nabla v\|_{L^{r}(B_{R}^{+}(0))}. 
\]
Next, since $s = r^{*}$, we have that $\frac{1}{r*} + \frac{1}{s'}  = 1$, and applying H\"older's inequality with exponents $r^{*}$ and $s'$  we obtain that 
\[
|I_{2}| \leq \int_{B_{R}^{+}(0))} |v||{\mathbb{A}}\nabla w| |\nabla \zeta|dx \leq C \|\nabla \zeta \|_{L^{\infty}} \|v\|_{L^{r^{*}}(B_{R}^{+}(0))} \|\nabla w\|_{L^{s'}(B_{R}^{+}(0))} \leq C \|v\|_{W^{1, r}(B_{R}^{+}(0) }.
\]
Finally, by Sobolev embedding,
\[
|I_{3}| \leq C \|\nabla \zeta\|_{L^{\infty}}\|{\bf g}\|_{L^{s^{'}}}\| v\|_{L^{r*}}  \leq C  \|v\|_{W^{1, r}(B_{R}^{+}(0))}. 
\]
Combining the above inequalities, we obtain that there exists a constant $C$, that depends on $\zeta$, such that 
\[
\left | \int_{B_{R}^{+}(0)}\zeta \nabla v \cdot {\bf g} dx  \right| \leq C \|v\|_{W^{1, r}(B_{R}^{+}(0))}. 
\]
As this holds for all ${\bf g} \in C_c^\infty(B_{R}^{+}(0))$ such that $\| {\bf g} \|_{L^{s'}} \leq 1$, by duality we get  
\[
\left ( \int_{B_{R}^{+}(0)}|\zeta \nabla v|^{r^*} dx  \right)^{\frac{1}{r^*}} \leq C \|v\|_{W^{1, r}(B_{R}^{+}(0))}. 
\]

Now,  given $\epsilon\in (0,1)$, we may choose $\zeta \in C_{c}^{\infty}(B_{R}(0))$  such that  $0\leq \zeta \leq 1$, $\zeta = 1$ on $B_{\epsilon R}(0)$. 
For this choice of $\zeta$ we have  in particular that $v\in W^{1, r^{*}} (B_{\epsilon R}^{+}(0)) $ and 
\[ \| v\|_{W^{1, r^{*} } (B_{\epsilon R}^{+}(0)) } \leq C_{\epsilon}  \|v\|_{W^{1, r}(B_{R}^{+}(0))}. \] 

It is easy to see that $r < r^{*}$. If $r^{*}  =2$, then we are done. Otherwise, applying the argument of {\bf Case 2} yields $r^{**} > r^{*}$ that $v \in W^{1, r^{**} }(B_{(\epsilon/2) R}^{+}(0))$, and then $v \in W^{1, r^{***} }(B_{(\epsilon/3)R}^{+}(0))$ and so on, until $r^{**\cdots *}$ reaches  the first value bigger than 2, at which point we apply {\bf Case 1} to obtain that $v \in W^{1, 2}(B_{(\epsilon/m)R}^{+}(0) )$ for some positive integer $m = m(r)$.  In particular, taking $\tau=\varepsilon/m $ gives the desired result. We emphasize that for the argument to work we need to verify that the solution $w$ belongs to $W^{1, r^{**\cdots*}}_{0} (B_{R}^{+}(0))$, and for that we must choose $\delta$ sufficiently small.  
\end{proof}

\end{document}